\documentclass[final]{siamart0216}

\usepackage{amsfonts}
\usepackage{graphicx}
\usepackage{epstopdf}
\usepackage{algorithmic}
\usepackage{comment}
\usepackage{amsmath}
\usepackage{pgfplots}
\usepackage{tikz}
\usetikzlibrary{positioning}
\usetikzlibrary{arrows}

\newcommand{\R}{\mathbb{R}}

\newcommand{\rd}{\mathrm{d}}
\newcommand{\matbra}[1]{\left[ \hspace{-3pt} #1 \hspace{-3pt} \right]}
\newcommand{\dthe}[2]{\theta^{\left( #1 \right)}_{ #2 }}
\newcommand{\xd}[2]{x^{\left( #1 \right)}_{ #2 }}
\newcommand{\ud}[2]{u^{\left( #1 \right)}_{ #2 }}
\newcommand{\bud}[2]{\overline{u}^{\left( #1 \right)}_{ #2 }}
\newcommand{\fd}{\delta^+}
\newcommand{\bd}{\delta^-}
\newcommand{\cd}{\delta^{\langle 1 \rangle}}
\newcommand{\cdd}{\delta^{\langle 2 \rangle}}
\newcommand{\fa}{\mu^+}
\newcommand{\ba}{\mu^-}

\newcommand{\caa}{\mu^{\langle 2 \rangle}}
\newcommand{\acd}{\delta^{-1}_{\mathrm{FD}}}
\newcommand{\tacd}{\tilde{\delta}^{-1}_{\mathrm{FD}}}
\newcommand{\relmiddle}[1]{\mathrel{}\middle#1\mathrel{}}

\newsiamremark{remark}{Remark}

\ifpdf
  \DeclareGraphicsExtensions{.eps,.pdf,.png,.jpg}
\else
  \DeclareGraphicsExtensions{.eps}
\fi

\newcommand{\TheTitle}{A self-adaptive moving mesh method for the short pulse equation via its hodograph link to the sine-Gordon equation} 
\newcommand{\TheAuthors}{S. Sato, K. Oguma, T. Matsuo, B.-F. Feng}

\headers{\TheTitle}{\TheAuthors}

\title{{\TheTitle}\thanks{This work was partly supported by JSPS KAKENHI Grand Numbers 26390126, 25287030, and 15H03635, and also by CREST, JST. The first author is supported by the Research Fellowship of the Japan Society for the Promotion of Science for Young Scientists.}}

\author{
Shun Sato\thanks{Graduate School of Information Science and Technology, The University of Tokyo, Bunkyo-ku, Tokyo, Japan
(\email{shun\_sato@mist.i.u-tokyo.ac.jp}, \url{http://www.sr3.t.u-tokyo.ac.jp/sato/}).}
\and
Kazuhito Oguma
\and
Takayasu Matsuo\thanks{Graduate School of Information Science and Technology, The University of Tokyo, Bunkyo-ku, Tokyo, Japan
(\email{matsuo@mist.i.u-tokyo.ac.jp}).}
\and
BaoFeng Feng\thanks{School of Mathematical and Statistical Sciences, The University of Texas Rio Grande Valley, Edinburg, TX, U.S.A.
(\email{baofeng.feng@utrgv.edu}).}
}

\usepackage{amsopn}

\ifpdf
\hypersetup{
  pdftitle={\TheTitle},
  pdfauthor={\TheAuthors}
}
\fi

\begin{document}
\maketitle


\begin{abstract}
The short pulse equation was introduced by Sch\"afer--Wayne (2004) 
for modeling the propagation of ultrashort optical pulses.
While it can describe a wide range of solutions, 
its ultrashort pulse solutions with a few cycles,
which the conventional nonlinear Schr\"odinger equation does not possess,
have drawn much attention.
In such a region, existing numerical methods turn out to require very
fine numerical mesh, and accordingly are computationally expensive.
In this paper, we establish a new efficient numerical method
by combining the idea of the hodograph transformation and 
the structure-preserving numerical methods.
The resulting scheme is a self-adaptive moving mesh scheme
that can successfully capture not only the ultrashort pulses
but also exotic solutions such as loop soliton solutions.
\end{abstract}



\pagestyle{myheadings}
\thispagestyle{plain}
\markboth{S. SATO, K. OGUMA, T. MATSUO, B.-F. FENG}{A SELF-ADAPTIVE MOVING MESH METHOD FOR THE SHORT PULSE EQUATION}

\section{Introduction}
\label{intro}

In this paper, we consider the numerical integration of the short pulse (SP) equation
\begin{equation}
u_{tx} = u + \frac{1}{6} \left( u^3 \right)_{xx}, 
\label{eq_spe}
\end{equation}
which is a model equation of ultrashort optical pulses in nonlinear media~\cite{SW2004,CJSW2005}. 
Here, $ t \in \R_{+} := [0, + \infty)$ and $ x \in \R $ denote temporal and spatial independent variables, respectively. 
The dependent variable $ u = u (t,x) $ represents the magnitude of the electric field, 
and subscripts $ t $ and $ x $ denote partial differentiations. 
Although this equation is usually considered 
on the whole real line $\R$ with the vanishing boundary condition
$|u(t,x)|\to 0$ as $|x|\to \infty$,
in this paper, for the consistency with numerical treatment,
we consider the problem on the circle, i.e., under the periodic boundary condition $ u(t,x+L) = u(t,x) $ for any $ t \in \R_+ $ and $ x \in \R $, 
where $L \in \R_+$ is a constant representing the length of the period.

We here summarize the theoretical results obtained so far, most of which were
established on $\R$.
Brunelli~\cite{B2005,B2006} found its hierarchy and bi-Hamiltonian structure. 
Sakovich--Sakovich~\cite{SS2005} proved the integrability of the SP equation by showing 
that it is associated with the sine-Gordon (SG) equation in light-cone coordinates
\begin{equation}
\theta_{ \tau s} = \sin \theta
\label{eq_sge}
\end{equation}
via a hodograph transformation. 
Here, $ \tau $ and $s$ denote new temporal and spatial independent variables.
It was then shown that breather solition solutions of the SP equation 
can be generated from those of the SG equation~\cite{SS2006}. 
Another interesting outcome of this SG formulation is that
wide range of multi-valued, exotic solutions such as loop solutions
can be regarded as the solutions of the SP equation; see, for example,
\cite{M2007,PA2009,LL2009,FCZML2010,GTGCC2015}. 
From the theoretical aspects as the partial differential equation (PDE),
the global well-posedness of the SP equation (and the SG equation) was discussed by Pelinovsky--Sakovich~\cite{PS2010}, 
whereas the local well-posedness was already proved by Sch\"afer--Wayne~\cite{SW2004}. 
A similar result for entropy solutions was recently shown by Coclite--di Ruvo~\cite{CR2015}.
On the other hand, Liu--Pelinovsky--Sakovich~\cite{LPS2009} found the conditions of the wave breaking in the SP equation (see also Sakovich~\cite{SakovichMT}).
Some generalizations of the SP equation have been considered as well, including
vector SP equation~\cite{PKB2008},
regularized SP equation~\cite{CMJ2009},
integrable coupled SP equation~\cite{F2012}, 
higher-order SP equation~\cite{KCS2013}, 
stochastic SP equation~\cite{KS2014}, 
and
(coupled) complex SP equation~\cite{F2015}.
There are also a few results for the circle (periodic) case.
Parkes~\cite{P2008} and Matsuno~\cite{M2008}, 
found various periodic solutions (see, also \cite{SWWKSF2010}). 
Liu--Pelinovsky--Sakovich~\cite{LPS2009} also dealt with local well-posedness and wave breaking scenario on the circle. 

When it comes to its numerical treatment, surprisingly,
it seems much less have been understood.
Sakovich~\cite{SakovichMT} employed the Fourier-spectral and Runge--Kutta methods 
to verify the well-posedness and wave-breaking of the SP equation. 
Pietrzyk--Kanatt\v{s}ikov~\cite{PK2015} applied a multi-symplectic integrator for the SP equation, 
and concluded that the multi-symplectic integrator is more efficient than Fourier-spectral methods. 

One reason for the lack of numerical studies
perhaps lies in the fact
that the SP equation intrinsically aims at (nearly) singular solutions.
Let us describe the situation taking the typical pulse solution
on $\R$ devised in~\cite{SS2006}:
\begin{equation}
\begin{cases}
{\displaystyle u ( \tau ,s ) = 4 \xi \zeta \frac{ \xi \sin \psi \sinh \phi  + \zeta \cos \psi \cosh \phi }{ \xi^2 \sin^2 \psi + \zeta^2 \cosh^2 \phi },}\\[10pt]
{\displaystyle x( \tau ,s) = s + 2 \xi \zeta \frac{ \xi \sin 2 \psi - \zeta \sinh 2 \phi }{ \xi^2 \sin^2 \psi + \zeta^2 \cosh^2 \phi }.}
\end{cases}
\label{eq_spe_breather}
\end{equation}
The parameter $ \xi \in (0,1)$ determines the shape of the wave, and other parameters are defined as 
\[ \zeta := \sqrt{1 - \xi^2} , \quad \phi  := \xi ( s + \tau ) , \quad \psi := \zeta (s - \tau). \]
This  solution with $ \xi < \xi_{\mathrm{cr}} = \sin \left( \pi / 8 \right) \approx 0.383 $ 
is nonsingular~(see, \cref{fig_shortpulse} for wave profile with $ \xi = 0.3 $), whereas the singularities $ u_x \to \pm \infty $ appear when $ \xi \ge \xi_{\mathrm{cr}}$; beyond the critical value,
the solution becomes multi-valued.
According to our preliminary numerical tests,
existing numerical methods can handle nonsingular solutions 
far enough away from the threshold $\xi_{\mathrm{cr}}$.
As $\xi$ approaches to the threshold, however,
those numerical methods tend to require very small mesh sizes,
and accordingly high computational cost.
This is a serious problem if we recall what the SP equation is for---it aims at capturing ultrashort optical pulses with only a few oscillations of the electric field. 
The rapid development of laser technologies enabled the actual generation 
of such ultrashort pulses.
In such a region, the standard nonlinear Schr\"odinger equation, which assumes slowly varying amplitude, does not give an accurate approximation, and that is exactly the situation the SP equation is designed for.
In fact, one may check the original paper~\cite[Figure~4]{SW2004} to find
that the main focus is on the extreme case where only three or four visible peaks are in a single pulse (see also \cite[Figures 2,3]{CJSW2005}).
In terms of the solution~\cref{eq_spe_breather},
the situation corresponds to nearly singular case
with $\xi \simeq \xi_{\mathrm{cr}}$; see \cref{fig_shortpulse} ($\xi=0.30$),
and \cref{fig_SPDVDM} ($0.38$).
Let us, for example, see a result by a special numerical integrator
that keeps an invariant of the SP equation 
(devised in \cref{subsec_sp_normal}).
Such a conservative method is generally believed to give qualitatively
better results than generic methods for numerically tough problems,
but still not sufficient for the SP equation.
We see in \cref{fig_SPDVDM} that the numerical solution rapidly develops
unphysical numerical oscillations.
The present authors observed that 
we need extremely fine meshes to eliminate such oscillations.

\begin{figure}[htbp]
\begin{minipage}{0.48\textwidth}
\pgfplotsset{width=5.5cm}
\centering
\begin{tikzpicture}
\begin{axis}[compat = newest,xlabel={$x$},ylabel={$u$},enlarge x limits=false,ymax=1.5,ymin=-1.5]
\addplot[black,smooth] table[x=x,y=y] {perbreatherCD0xy.dat};
\end{axis}
\end{tikzpicture}
\caption{The smooth pulse solution~\cref{eq_spe_breather} with the parameter $ \xi = 0.3 $. }
\label{fig_shortpulse}
\end{minipage}
\hspace{3pt}
\begin{minipage}{0.48\textwidth}
\centering
\pgfplotsset{width=5.5cm,compat=newest}
\begin{tikzpicture}
\begin{axis}[
		xlabel=$x$, 
		ylabel=$t$,ytick={0,5,10},
		zlabel=$u$, 
		view={60}{45},
		enlarge x limits=false]
\addplot3[black] table {stdDVDMpulse0.dat};
\addplot3[black] table {stdDVDMpulse1.dat};
\addplot3[black] table {stdDVDMpulse2.dat};
\addplot3[black] table {stdDVDMpulse3.dat};
\addplot3[black] table {stdDVDMpulse4.dat};
\addplot3[black] table {stdDVDMpulse5.dat};
\addplot3[black] table {stdDVDMpulse6.dat};
\addplot3[black] table {stdDVDMpulse7.dat};
\addplot3[black] table {stdDVDMpulse8.dat};
\addplot3[black] table {stdDVDMpulse9.dat};
\addplot3[black] table {stdDVDMpulse10.dat};
\end{axis}
\end{tikzpicture}
\caption{The numerical solution of the norm-preserving numerical method~\cref{eq_std_dvdm} under the initial condition~\cref{eq_spe_breather} with $ \xi = 0.38 $ $( \Delta x = L/K = 66.96/127 \approx 0.527 , \ \Delta t = 0.1)$.}
\label{fig_SPDVDM}
\end{minipage}
\end{figure}

A good approach to treat such a localized, nearly singular solutions
is to employ the
\emph{moving mesh methods} (see, e.g., \cite{HR2011,BHR2009}). 
In a moving mesh method~\cite{HR2011},
first the original PDE on $u(t,x)$ is transformed to a
parametrized form on, say, $u(\tau,s)$ and $x(\tau,s)$
($\tau, s$ are new independent variables that replace $t, x$),
and it is solved with a ``moving mesh PDE'' on $x(\tau,s)$
that describes the moving mesh strategy.
Although this approach is quite effective in many problems,
the application to the SP equation seems a tough task.
The difficulty lies on the term $ u_{tx} $, which implicitly casts
a constraint $ \int^L_0 u(t,x) \rd x = 0 $ (see, for example, Horikis~\cite{H2009}) on all solutions.
This can be easily seen by integrating the SP equation.
This is just a linear constraint, and can be satisfied by simple methods (for example, the Fourier-spectral method~\cite{SakovichMT}).
It becomes, however, highly nontrivial when we consider a moving mesh method---the parametrization destroys the simple constraint, and its treatment becomes a big mathematical challenge; see \cref{rem_tradeoff}.
As a preliminary test, we considered a discretization of the parametrized form of the SP equation~\cref{eq_spemg} by the Fourier-spectral method, and tried to solve it by `ode15s' of MATLAB 2013R.
The trial was not satisfactory, which showed that a careful treatment 
of the implicit constraint is indispensable.
In fact, to the best of the present authors' knowledge, 
there have been no moving mesh schemes for the PDEs of the form $ u_{tx} = f(t,u)$ in the literature, which 
should be attributed to the observation above.
Here, we do not deny the possibility in this research direction,
since such PDEs include various practical applications such as
the Ostrovsky equation~\cite{O1978} (and its generalizations), 
the nonlinear Klein--Gordon equation in light cone coordinates, among others,
and thus the new techniques in this direction will surely open a new door.
Still, in the present paper,
we like to leave it to future works, and turn to a different approach.

In this paper, we attempt to import some techniques from
the research field of \emph{integrable systems}. 
Feng--Inoguchi--Kajiwara--Maruno--Ohta~\cite{FIKMO2011} 
made use of the fact that the SP equation is related to the SG equation by the hodograph
transformation.
They started with a known integrable discretization of the SG equation,
and by translating it with a discrete hodograph transformation,
reached a new integrable discretization of the SP equation.
(Feng et al.~\cite{FIKMO2011} also dealt with the Wadati--Konno--Ichikawa elastic beam equation~\cite{WKI1979} and the complex Dym equation; 
see also~\cite{FMO2010_1,FMO2010_2,FMO2010_3,OMF2008,FMO2014,FCCMO2015}.)
This approach had an advantage that the sample points automatically distributed uniformly along the solution curve, and actually 
it beautifully reproduced various solition solutions, either single- or multi-valued.
Despite the success, however, there remained some issues to regard
the approach as a practical numerical method.
First, as the studies of integrable systems, the works above
mainly focused on the discrete reproduction of the solitons,
and it was not clear how general initial data could be treated.
Second, they considered problems only on $\R$, and the treatment of
boundaries was set outside their scope.
This means the numerical integration should be limited only up to
a short time before the pulse reaches a boundary of the domain,
which is unsatisfactory as a practical numerical scheme.

In this paper, we provide a new comprehensive numerical framework
for the SP equation, inheriting the idea of hodograph transformation
from the integrable systems approach.
To this end, we have to clarify what the transformation is,
in the language of numerical analysis.
We will actually point out that 
the hodograph transformation between the SP and the SG equations can be 
regarded as the combination of ``parametrization by arc-length $s$'' and ``argument transformation'' (see, \cref{sec_trans}).
In this sense, the hodograph transformation approach can be understood
as an alternative way of the standard moving mesh method.
Then we will show how the periodic boundary condition can be treated.
This involves several nontrivial tasks;
first, the description in the transformed SG equation itself is not straightforward (\cref{prop_pbc}).
Second, the SP and the SG equations are closely related, and 
the original periodic boundary condition in the SP equation corresponds to
an invariant preservation in the SG equation (\cref{lemL}).
Thus, for a numerical computation to make sense,
the SG equation should be solved with a special integrator that replicates the invariant.
In total, in \cref{sec_trans,sec_scheme}, we will show how 
deliberate 
consideration is necessary to build up a consistent framework.

The resulting scheme has the following advantages:
\begin{enumerate}
\item[(i)] it can deal with arbitrary initial conditions,
\item[(ii)] it is a self-adaptive moving mesh method, i.e.,
the grids distribute appropriately without using an extra mesh PDE,
\item[(iii)] the discrete periodic boundary condition of the SP equation is rigorously guaranteed by preserving an invariant in the SG equation, 
\item[(iv)] the computation of the SG equation is quite stable
 thanks to a recently developed numerical integrator in \cite{FSM2016+}, and
\item[(v)] the scheme can deal with multi-valued solutions as well; in particular, it beautifully works also for single-valued solutions near singularity.
\end{enumerate}

The rest of the present paper is organized as follows. 
In \cref{sec_trans}, 
we review and discuss the relation between the SP and the SG equations.
Some facts are from the literature, but 
we will carry out our argument more carefully clarifying 
the treatment of the boundaries.
\Cref{sec_scheme} is devoted to the proposed scheme for the SP equation on the periodic domain. 
Then, we show the superiority of our proposed method by numerical experiments in \cref{sec_ne}. 
Finally, \cref{sec_conclusion} is devoted to some concluding remarks.


\section{Relating the SP equation to the SG equation: continuous case}
\label{sec_trans}

In this section, we summarize some facts about the SP and the SG equations.
We start by pointing out some properties of the SP equation in \cref{subsec_arc}.
In \cref{sec_equiv}, we discuss the transformation between the SP and the SG equations,
and state exactly in which sense they are equivalent.
Then \cref{sec_fixed} is devoted to the proof
of an important, non-trivial fact that
the physical and computational spaces can be fixed in our regime.
Until this point, we leave the boundary condition open,
so that the mathematical results so far are applicable to general
cases other than the periodic boundary condition.
In \cref{subsec_bound_cond}, we focus on the circle case,
and show how we can treat the periodic boundary condition.
Finally, in \cref{sec_implicit}, we show our framework
successfully recovers the implicit constraint of the SP equation,
which completes the whole picture.

\subsection{The SP equation: its conservation law and Hamiltonian structures}
\label{subsec_arc}

We start by reviewing a certain conservation law
that will play an essential role in the following discussion.
More precisely, the solution of the SP equation~\cref{eq_spe} 
preserves the ``total arc-length''~\cite{B2005}. 

Let us consider the SP equation on a fixed physical domain $[0,L]$ (we leave the boundary condition open for the moment).
For a smooth function $ u : [0,L] \to \R $, the total arc-length $ \mathcal{S} (u) $ is defined as 
\[ \mathcal{S} (u) := \int^L_0 \sqrt{ 1 + \left( u_x (x) \right)^2 } \rd x.  \]
For a function $ u : \R_+ \times [0,L] \to \R $ and $ t \in \R_+ $,  
the notation $ u(t) : [0,L] \to \R $ denotes a function satisfying $ \left( u (t) \right) (x) = u(t,x) $ for any $ x \in [0,L] $. 
As shown in the following proposition, 
the total arc length $ \mathcal{S} (u (t)) $ of a solution $ u $ of the SP equation~\cref{eq_spe} 
is invariant under certain boundary conditions.  

\begin{proposition}
Let $ u : \R_+ \times [0,L] \to \R $ be a smooth solution of the short pulse equation~\cref{eq_spe} with a boundary condition satisfying  
\begin{equation*}
\left[ \frac{ \left( u (t,x) \right)^2 \sqrt{ 1 + \left( u_x (t,x) \right)^2 }  }{2} \right]^L_0 = 0.
\label{eq_arclength_boundary}
\end{equation*}
Then, $ ( \rd / \rd t ) \mathcal{S} (u (t) ) = 0 $ holds for any $ t \in \R_+ $. 
\label{prop_arclength}
\end{proposition}

\begin{proof}
This proposition can be shown in a straightforward way: 
\begin{align*}
\frac{\rd}{\rd t } \mathcal{S} ( u (t) ) 
&= \int^L_0 \frac{ \partial }{ \partial t } \sqrt{ 1 + u^2_x } \rd x 
= \int^L_0 \frac{u_x }{ \sqrt{ 1 + u^2_x } } u_{tx} \rd x \\
&= \int^L_0 \frac{u_x }{ \sqrt{ 1 + u^2_x (t,x) } } \left( u + \frac{1}{6} \left(  u^3 \right)_{xx} \right) \rd x 
= \left[ \frac{ u^2 \sqrt{ 1 + u^2_x }  }{2} \right]^L_0 =0.
\end{align*}
The last equality holds due to the assumption on the boundary condition. 
\end{proof}

We also recall the fact that the SP equation has 
the bi-Hamiltonian structure (Brunelli~\cite{B2006}):
\begin{align}
u_t &= \left( \partial^{-1}_x + u_x \partial_x^{-1} u_x \right) \frac{\delta \mathcal{I}}{\delta u}, &
\mathcal{I} (u) &= \frac{1}{2} \int^L_0 u^2  \rd x ,
\label{eq_spe_vfnorm}\\
u_t &= \partial_x \frac{\delta \mathcal{E}}{\delta u},&
\mathcal{E} (u) &= \int^L_0 \left( \frac{1}{24} u^4 - \frac{1}{2} \left( \partial^{-1}_x u \right)^2  \right) \rd x,
\label{eq_spe_vfenergy}
\end{align}
where the antiderivative $ \partial_x^{-1} $ is defined as 
\begin{equation*}
\partial_x^{-1} v (x) = \int^x_0 v (y ) \rd y - \frac{1}{L} \int^L_0 \int^z_0 v(y) \rd y \rd z , 
\label{eq_def_anti}
\end{equation*}
which is employed by Hunter~\cite{H1990} for the short wave equation. 
The Hamiltonian structure can be utilized to construct 
structure-preserving schemes directly applying to the SP equation
(see \cref{subsec_sp_normal}).

\subsection{Transformation from the SP equation to the SG equation}
\label{sec_equiv}

Now let us discuss how we transform the SP equation to the SG equation.
The transformation starts with the standard procedure
in the moving mesh methods (see, for example,~\cite{HR2011}).
We introduce new space and time variables
 $ s \in [0,S]$ and $ \tau \in \R_+ $, 
($ \tau $ actually coincides with $ t$, 
but it is introduced here for the ease of handling independent variables). 
For the moment, $S$ is assumed to be some constant, and later 
it will be redefined (see \cref{fig_sg_diagram} and its explanation).
The differential operators $ \partial / \partial x $ and $ \partial / \partial t $ can be replaced
 with $ \partial / \partial s  $ and $ \partial / \partial \tau $
using the identities
\begin{align*}
\frac{ \partial }{ \partial x } &= \frac{1}{x_s} \frac{ \partial }{ \partial s }, &
\frac{ \partial }{ \partial t } &= \frac{ \partial }{ \partial \tau } - \frac{ x_{\tau} }{ x_s } \frac{ \partial }{ \partial s }.
\label{eq_trans_mg}
\end{align*}
Then the SP equation~\cref{eq_spe} is transformed to
\[ \left( \frac{1}{x_s} \frac{ \partial }{ \partial s } \right) \left( \frac{ \partial }{ \partial \tau } - \frac{ x_{\tau} }{ x_s } \frac{ \partial }{ \partial s } \right) u = u + \frac{1}{6} \left(  \frac{1}{x_s} \frac{ \partial }{ \partial s } \right)^2 u^3, \]
which can be further simplified into
\begin{align}
x_s \left( u_{ \tau s} x_s - x_{\tau s} u_s \right) 
=  u x_s \left( x_s^2 + u_s^2 \right) + \left( x_{\tau} + \frac{u^2}{2} \right) \left( u_{ss} x_s - x_{ss} u_s \right). 
\label{eq_spemg}
\end{align}
We call the equation~\cref{eq_spemg} a ``parametrized form'' of the SP equation. 
We call the solution of~\cref{eq_spemg} as the solution curve of the SP equation by convention. 

Now we here start something different from the standard 
moving mesh approach.
We regard the variable $s$ as the ``arc-length'' of the
solution $u(t,x)$, and employ
the hodograph transformations below~\cite{SS2005}:
\begin{subequations}\label{eq_hodograph}
\begin{equation}
\matbra{ \begin{array}{c} x( \tau ,s) \\ u( \tau ,s) \end{array} } = \matbra{ \begin{array}{c} x( \tau ,0) \\ u( \tau ,0) \end{array} } 
+ \int^s_0 \matbra{ \begin{array}{c} \cos \theta ( \tau , \sigma ) \\ \sin \theta ( \tau ,\sigma) \end{array} } \rd \sigma,
\label{ht}
\end{equation}
\begin{equation}
\matbra{ \begin{array}{c} x_s (\tau ,s) \\ u_s ( \tau ,s) \end{array} } = \matbra{ \begin{array}{c} \cos \theta ( \tau , s ) \\ \sin \theta (\tau ,s) \end{array} }. 
\label{iht}
\end{equation}
\end{subequations}
The hodograph transformation connects 
the expression $ (x( \tau ,s) , u(\tau,s) )$ in the two dimensional Cartesian coordinates 
with another expression in terms of the argument $ \theta (\tau ,s) $. 
Computing $ \theta (\tau,s) $ from $ (x(\tau,s) ,u(\tau,s) )$ can be done by~\cref{iht}. 
However, we need the ``base point'' $ (x(\tau,0), u(\tau,0)) $ when 
we compute $ ( x(\tau,s) ,u(\tau,s) )$ back from $ \theta(\tau,s) $. 
Notice that there remains a degree of freedom in choosing such a
base point.

Let us relate the SP equation to the SG equation using the new variables.
For any solution $ \left( x(\tau,s) , u (\tau,s) \right) $ of the parametrized form~\cref{eq_spemg} of the SP equation, 
by substituting~\cref{eq_hodograph}, we obtain 
\begin{equation}
\left( \theta_{\tau} - u \right) \cos \theta = \left( x_{\tau} + \frac{u^2}{2} \right) \theta_s . 
\label{spe2sge1}
\end{equation}
Note that the relation above is equivalent to the equation~\cref{eq_spemg}. 
By differentiating ~\cref{spe2sge1} with respect to $ s $, we have
\begin{equation}
\left( \theta_{\tau s} - \sin \theta \right) \cos \theta = \left( x_{\tau} + \frac{u^2}{2} \right) \theta_{ss} . 
\label{spe2sge2}
\end{equation}
From \cref{spe2sge1} and \cref{spe2sge2}, we see
\[ \left\{ \left( \theta_{\tau} - u \right) \theta_{ss} - \left( \theta_{\tau s} - \sin \theta \right) \theta_s \right\} \cos \theta = 0. \]
If $ \cos \theta \neq 0 $ and $ \theta_s \neq 0 $ hold almost everywhere, we obtain 
\[ \left( \frac{ \theta_{\tau} - u }{ \theta_s } \right)_s = 
\frac{\left( \theta_{\tau} - u \right) \theta_{ss} - \left( \theta_{ \tau s} - \sin \theta \right) \theta_s}{ \theta_{s}^2 } = 0. \]
If, furthermore, 
$\theta_{\tau} (\tau,0) = u(\tau,0)$
and $\theta_s (\tau ,0) \neq 0$ holds, we obtain
\[ \frac{ \theta_{\tau} (\tau,s) - u (\tau,s) }{ \theta_s (\tau,s) } = \frac{ \theta_{\tau} (\tau,0) - u (\tau,0) }{ \theta_s (\tau,0) } = 0,  \]
which implies $ \theta_{\tau} - u = 0 $.
It in turn implies
\begin{equation}
x_{\tau} + \frac{u^2}{2} = 0 \label{eq_condxy}
\end{equation}
from~\cref{spe2sge1}.
Finally, by~\cref{spe2sge2}, we conclude $ \theta_{\tau s} - \sin \theta = 0$,
which is the desired sine-Gordon equation~\cref{eq_sge}.
Below, we consider the equivalence of the SP and the SG equations more precisely.

Before that, at this point let us show the overall picture of our strategy
in \cref{fig_sg_diagram}.
By ``physical domain'' we mean the SP equation on $u(t,x)$ or
$\left( x(t,s), u(t,s) \right)$.
As seen above, the SP equation is transformed to the SG equation on $\theta(\tau,s)$,
which is called ``computational domain.''
Our main strategy is to solve the SG equation (instead of the SP equation),
and transform the solution back to the physical space.
Note that we denote the length of the physical and computational domains
by $L(\tau)$ and $S(\tau)$, 
stating the possibility of their dependence on time $\tau$.
The reason of this is because since we regard the variable
$s$ as the arc-length of the solution, 
$S(\tau)$ can vary in time in general (i.e., in general PDEs)
if we fix $L$;
similarly, $L(\tau)$ is allowed to depend on time if we fix $S$.
It should be mathematically proved that they remain constants
in the case of the SP and the SG equations.

\begin{remark} \label{rem_tradeoff}
Let us briefly comment on the difference between the standard
moving mesh method and the approach here.
In the standard method, after obtaining the parametrized PDE,
one would furnish it
with what we call ``a moving mesh PDE (MMPDE),'' which controls the grids.
An advantage of this approach is that 
we can simply fix $L$ and $S$ as constants.
Instead, there arises a difficulty that
it is hard to find a counterpart of the implicit constraint 
$\int_0^L u(t,x){\rm d}x=0$ in \cref{eq_spemg}.
In contrast, in the approach here, 
the constraint can be replicated as we will see later,
but a careful discussion is necessary to guarantee
that $L(\tau)$ and $S(\tau)$ stay fixed.
In this sense, there is a trade-off between the two approaches.

Also note that, in view of \cref{prop_spe2sge} below, 
our method can be regarded as a special case of moving mesh methods 
whose moving mesh PDE is $ x_{\tau} = - u^2/2 $,
which implies the mesh is distributed uniformly with respect to arc-length. 
However, unlike general moving mesh methods, we conducted
further transformation with respect to $\theta$ and 
combine the system of PDEs (the SP equation and MMPDE) into a single PDE (the SG equation). 
\end{remark}

\begin{figure}[htbp]
\centering
\begin{tikzpicture}[line width= 0.5pt,xscale = 0.8,yscale = 1]
\node at ( -4,0) {Physical domain};
\node at ( -3.9, -0.4) {(SP)};
\draw[->] (-6,-1) -- (-2,-1);
\node at (-1.7,-1) {$x$};
\draw (-5.5,-1.2) -- (-5.5,-0.8);
\node at (-5.5,-1.5) {$0$};
\draw (-2.5,-1.2) -- (-2.5,-0.8);
\node at (-2.5,-1.5) {$L(\tau )$};
\node at (4,0) {Computational domain};
\node at (4.1,-0.4) {(SG)};
\draw[->] (2,-1) -- (6,-1);
\node at (6.3,-1) {$s$};
\draw (5.5,-1.2) -- (5.5,-0.8);
\node at (5.5,-1.5) {$S(\tau)$};
\draw (2.5,-1.2) -- (2.5,-0.8);
\node at (2.5,-1.5) {$0$};
\node at (-4,-2) {$u(0,x)$};
\draw[<->] (-4,-2.3) -- (-4,-2.7);
\draw[loosely dashed] (-6,-2.5) -- (6,-2.5);
\node at (-4,-3) {$ (x(0,s) , u(0,s) )$};
\draw (-5.9,-1.8) -- (-6,-1.8) -- (-6,-3.2) -- (-5.9,-3.2);
\node at (-6.7,-2.5) {$t=0$}; 
\draw[->] (-2,-3) -- (2,-3);
\node at (0,-3.5) {hodograph trans.};
\node at (4,-3) {$\theta (0,s) \ (S(0) := S)$};
\draw[->] (4,-3.3) -- (4,-3.7);
\node at (4,-4) {$ \theta (\tau,s) $};
\draw[->] (2,-4) -- (-2,-4);
\node at (-4,-4) {$ ( x(\tau,s) , u (\tau,s)) $};
\node at (-4,-5) {$ L (\tau) = \mbox{const.}$};
\draw (-1.4,-5.05) -- (1.4, -5.05);
\draw (-1.4,-4.95) -- (1.4,-4.95);
\draw (1.35,-4.9) -- (1.45,-5) -- (1.35,-5.1);
\node at (0,-4.7) {\cref{prop_arclength}};
\node at (4,-5) {$ S(\tau) = \mbox{const.}$};
\node at (-4,-6) {$ L (\tau) = \mbox{const.}$};
\draw (-1.4,-6.05) -- (1.4, -6.05);
\draw (-1.4,-5.95) -- (1.4,-5.95);
\draw (-1.35,-5.9) -- (-1.45,-6) -- (-1.35,-6.1);
\node at (0,-5.7) {\cref{propai}};
\node at (4,-6) {$ S(\tau ) = \mbox{const.}$};
\end{tikzpicture}
\caption{The diagram of our strategy for self-adaptive moving mesh integration. }
\label{fig_sg_diagram}
\end{figure}
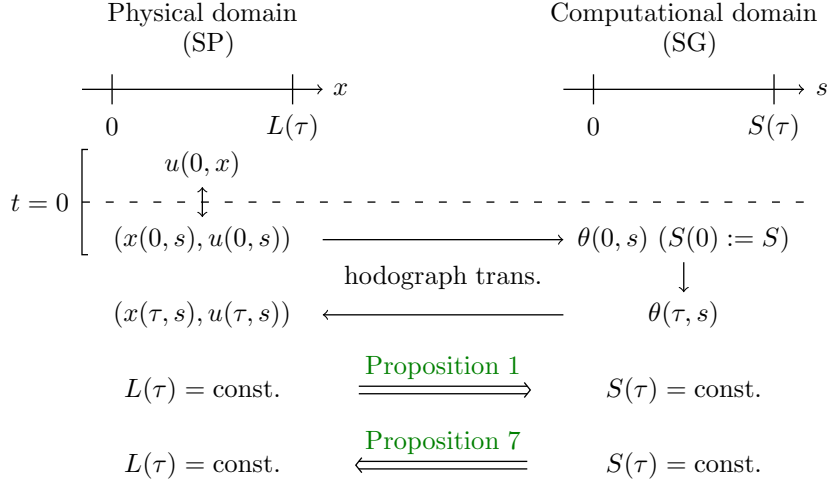


The discussion above has been already roughly done
in the literature, but here we emphasize that the above reveals
necessary constraint on the base point (see also \cref{rem:R}).
Still, we should note that the deformations in the discussion
above are not always mathematically rigorously equivalent,
and we have to be careful
exactly in which sense the solution
of one equation solves the other.
It seems this point as well has not been explicitly discussed 
in the literature.

We first show that  any solution of the SG equation~\cref{eq_sge}
solves the SP equation 
if the hodograph transformation satisfies the additional constraint 
\begin{equation}
x_{\tau} (\tau,0) = - \frac{\left(\theta_{\tau} (\tau,0) \right)^2}{2}, \qquad u (\tau,0) = \theta_{\tau} ( \tau ,0). 
\label{abp}
\end{equation}
Note that here we leave the boundary condition open,
and by ``solution'' we mean any functions that satisfy the equation 
on the specified interval.

\begin{proposition}
Let $ \theta $ be a solution of the SG equation~\cref{eq_sge}
on $s\in[0,S(\tau)]$, 
and $ \left( x , u \right) $ be a curve obtained via the hodograph transformation~\cref{ht} from $ \left( x(\tau,0), u(\tau,0) \right) $ 
satisfying~\cref{abp}. 
Then, $ \left( x, u \right)$ is a solution of the parametrized form~\cref{eq_spemg} of the SP equation on $x\in[0,L(\tau)]$. 
\label{prop_sge2spe}
\end{proposition}

\begin{proof}
From the condition~\cref{abp}, we see
\begin{align}
x_{\tau} (\tau,s) &= x_{\tau} (\tau,0 ) + \int^s_0 \left( \cos \theta (\tau, \sigma) \right)_{\tau} \rd \sigma = x_{\tau} (\tau,0) - \int^s_0  \sin \theta ( \tau , \sigma) \theta_{\tau} ( \tau ,\sigma ) \rd \sigma \notag \\
&=x_{\tau} (\tau,0) - \int^s_0 \theta_{\tau s} ( \tau , \sigma) \theta_{\tau} (\tau ,\sigma) \rd \sigma = x_{\tau} (\tau,0) - \left[ \frac{ \left( \theta_{\tau} ( \tau ,\sigma) \right)^2 }{2} \right]_0^s \notag \\
&= - \frac{ \left( \theta_{\tau} (\tau ,s) \right)^2 }{2}, \notag \\
u (\tau ,s)  &= u (\tau ,0) + \int^s_0 \theta_{\tau s} (\tau ,\sigma) \rd \sigma = u (\tau ,0) + \left[ \theta_{\tau} (\tau, \sigma) \right]^s_0 = \theta_{\tau} (\tau,s), \label{sg2sp_cond}
\end{align}
and therefore we obtain~\cref{eq_condxy}.
This and \cref{sg2sp_cond} imply the relation~\cref{spe2sge1} holds. 
Hence, $ (x,u ) $ is a solution of~\cref{eq_spemg}. 
\end{proof}

Furthermore, we can prove the converse under the assumption $ x_s \neq 0 $, i.e., as far as the solution is single-valued. 
This is related to the danger of using the parametrized form~\cref{eq_spemg}
when $x_s = 0$ can happen;
this corresponds to the ``mesh entanglement'' issue
in the context of the standard moving mesh method~\cite{BHR2009}.

\begin{proposition}
Let $ u , x $ be a solution of the parametrized SP equation~\cref{eq_spemg} 
on $x\in[0,L(\tau)]$
satisfying the condition~\cref{eq_condxy}.
We assume $ x_s ( \tau ,s ) \neq 0 $ holds for any $ \tau \in \R_+ $ and $ s \in \R $. 
Then, the parameter $ s$ represents arc-length, i.e., the hodograph transformation~\cref{ht} and \cref{iht} make sense and $ \theta $ is uniquely determined. 
Furthermore, $ \theta $ is a solution of the SG equation~\cref{eq_sge}
on $s\in[0,S(\tau)]$. 
\label{prop_spe2sge}
\end{proposition}

\begin{proof}
Under the assumption $ x_s \neq 0 $, 
by substituting \cref{eq_condxy} and $ x_{ \tau s } = - u u_s $ (differentiation of \cref{eq_condxy} with respect to $s$) into \cref{eq_spemg}, we see $u_{ \tau s } = u x_s$.
Thus, we see 
\begin{equation*}
\frac{\partial}{\partial \tau } \left( x_s^2 + u_s^2 \right) 
= 2 \left( x_s x_{s \tau } + u_s u_{s \tau} \right)
= 2 \left( x_s \left( - u u_s \right) + u_s \left( u x_s \right) \right) = 0.
\end{equation*}
Hence, the parameter $ s $ represents arc-length 
and the former part of the proposition holds true. 
Moreover, since the solution of \cref{eq_spemg} satisfies \cref{spe2sge2} and we assume 
$ x_{\tau} + u^2/2 = 0 $, we see $ ( \theta_{\tau s} - \sin \theta ) \cos \theta = 0 $. 
Finally, this equality and the assumption $ x_s = \cos \theta \neq 0 $ implies that $ \theta_{\tau s} = \sin \theta $ holds.
\end{proof}

The discussion above (\cref{prop_sge2spe,prop_spe2sge}) 
rigorously justifies our strategy: 
we compute the numerical solution of the SG equation, 
and then we transform the solution by some discrete counterpart of the hodograph transformation~\cref{ht}.

\subsection{The fixed physical and computational domains}
\label{sec_fixed}

Next we prove that the lengths of the physical and computational
domains are kept constant when the boundary condition is appropriate.

From \cref{prop_arclength}, it is easy
to see that $S(\tau) = \mathcal{S}(u(\tau))$ does not depend on time
under typical boundary conditions (including the periodic 
case) if $L(\tau)$ is fixed.

Its converse requires more careful discussion.
To this end, it is essential to recall the fact that
the SG equation~\cref{eq_sge} can be written in the variational form 
\begin{align}
\frac{\partial^2}{\partial \tau \partial s } \theta &= \frac{ \delta \mathcal{H} }{ \delta  \theta }, &
\mathcal{H} (\theta) &= - \int^S_0 \cos \theta (\tau,s)  \rd s.
\label{sgevf}
\end{align}

The value of the functional $ \mathcal{H} $ is kept invariant 
under certain boundary conditions.

\begin{proposition}
\label{prophp}
Under a boundary condition satisfying
\begin{equation}
\left[ \frac12 \left( \theta_{\tau} (\tau,s) \right)^2 \right]^S_0 = 0, 
\label{sgbcond}
\end{equation}
the solution of the variational form~\cref{sgevf} satisfies $ ( \rd / \rd \tau ) \mathcal{H} ( \theta( \tau) ) = 0 $ for any $ \tau \in \R_+ $. 
\end{proposition}

\begin{proof}
It can be shown by using~\cref{sgevf} as follows: 
\begin{align*}
\frac{\rd}{\rd \tau} \mathcal{H} (\theta) 
= \int^S_0 \frac{ \delta \mathcal{H} }{\delta \theta } \theta_{\tau} \rd s 
= \int^S_0 \theta_{\tau s} \theta_t \rd s 
= \left[ \frac{1}{2} \left( \theta_{\tau} (\tau,s) \right)^2 \right]^S_0
= 0. 
\end{align*}
\end{proof}

The following lemma shows that the ``window size in the physical domain,''
i.e., the value of $L(\tau)$,
is given by the functional $\mathcal{H}(\theta)$.

\begin{lemma} \label{lemL}
Suppose that $ x,u : \R_+ \times [0,S] \to \R $ are obtained from $ \theta : \R_+ \times [0,S] \to \R $ by~\cref{ht}. 
Then, the relation $ \mathcal{L} \left( x(\tau) , u(\tau) \right) = - \mathcal{H} \left( \theta(\tau) \right) $ holds for any $ \tau \in \R_+ $, 
where the functional $ \mathcal{L} : C[0,S] \times C[0,S] \to \R  $ is defined as 
$ \mathcal{L} \left( f , g \right) := f(S) - f(0) $
for $ f,g \in C[0,S] $. 
\end{lemma}

Thanks to \cref{lemL,prophp}, 
the window size remains constant as far as the boundary condition is 
appropriate.

\begin{proposition}
Let $ \theta : \R_+ \times [0,S ] \to \R $ be the solution of the SG equation~\cref{eq_sge} satisfying the condition~\cref{sgbcond}. 
If $ x,u : \R_+ \times [0,S] \to \R $ is obtained from $ \theta $ by~\cref{ht}, 
then $ (\rd / \rd \tau) \mathcal{L} ( x(\tau) , u (\tau) ) = 0 $ holds for any $ \tau \in \R_+ $. 
\label{propai}
\end{proposition}

\subsection{Treatment of the periodic boundary condition}
\label{subsec_bound_cond}

So far we have left the boundary condition open.
Now let us consider the particular case, the periodic boundary
condition, to complete our framework in the continuous case.

The periodic boundary condition $ u(t,x+L ) = u(t,x) \ (\forall t \in \R_+ , \  \forall x \in \R )$ for the original form of the SP equation~\cref{eq_spe} corresponds to the boundary conditions 
\begin{equation}
x( \tau ,s+S) = x( \tau ,s) + L , \quad u ( \tau ,s+S) = u( \tau ,s) \qquad (\forall \tau \in \R_+ , \ \forall s \in \R ),
\label{eq_pb}
\end{equation}
for the parametrized form~\cref{eq_spemg}. 
The following proposition provides the transformation of 
this to the world of $\theta$.
Recall that, thanks to \cref{prop_sge2spe}, the solution curve obtained from that of the SG equation 
is a solution of the (parametrized) SP equation. 
Hence, the following proposition tells us the 
correspondence between the SP equation on the periodic domain and the SG equation on the (almost) periodic domain. 

\begin{proposition} \label{prop_pbc}
Let $ \theta (\tau ,s) $ be the solution of the sine-Gordon equation~\cref{eq_sge} 
under the boundary condition
\begin{equation}
\theta (\tau,s+S) = \theta (\tau,s) + 2 n \pi \qquad ( \forall s \in [0,S] ) 
\label{eq_apb}
\end{equation}
for some $ n \in \mathbb{Z} $. 
If $ ( x(\tau,s) , u (\tau,s) ) $ is the curve obtained from $ \theta (\tau,s) $ 
via the hodograph transformation~\cref{ht} 
with $ ( x(\tau,0), u(\tau,0) ) $ satisfying~\cref{abp}, 
then the periodic boundary condition~\cref{eq_pb} holds. 
\end{proposition}

\begin{proof}
From~\cref{eq_apb} and the definition of the hodograph transformation~\cref{ht}, we obtain 
\begin{align*}
\matbra{ \begin{array}{c} x( \tau,s+S) \\ u(\tau,s+S) \end{array} }  - \matbra{ \begin{array}{c} x(\tau,s) \\ u(\tau,s) \end{array} } 
&= \int^{s+S}_s \matbra{ \begin{array}{c} \cos \theta ( \tau , \sigma ) \\ \sin \theta ( \tau ,\sigma) \end{array} } \rd \sigma 
= \int^S_0 \matbra{ \begin{array}{c} \cos \theta (\tau , \sigma ) \\ \sin \theta (\tau,\sigma) \end{array} } \rd \sigma. 
\end{align*}
Hence, the boundary condition~\cref{eq_pb} is 
equivalent to the equation
\begin{equation*}
\int^S_0 \matbra{ \begin{array}{c} \cos \theta (\tau, \sigma ) \\ \sin \theta (\tau,\sigma) \end{array} } \rd \sigma = \matbra{ \begin{array}{c} L \\ 0 \end{array} }. 
\end{equation*}
The equality above can be verified by the simple calculation: 
\begin{align}
\int^S_0 \cos \theta (\tau, \sigma) \rd \sigma &= - \mathcal{H} \left( \theta (\tau) \right) = - \mathcal{H} \left( \theta (0) \right) = x(0,S) - x(0,0) = L,  \label{eq_cos_H} \\
\int^S_0 \sin \theta (\tau, \sigma ) \rd \sigma &= \int^S_0 \theta_{\tau s} (t, \sigma ) \rd \sigma = \left[ \theta_{\tau} (\tau, \sigma) \right]^S_0 = 0.  \label{eq_proof_y}
\end{align}
Here, we used \cref{prophp} and the boundary condition~\cref{eq_apb}. 
\end{proof}

Note that, the integer $ n $ is determined by the initial condition $ \theta (0,s) $, 
and remains constant along time evolution 
thanks to the continuity of the solution. 
Note also that the condition~\cref{eq_apb} satisfies~\cref{sgbcond},
and thus \cref{propai} holds;
i.e., $L$ does not depend on time.
This is extremely important, since ``periodic solution'' would lose
its meaning if $L(\tau)$ were changing.

\subsection{Recovery of the implicit constraint} \label{sec_implicit}

Since $L$ successfully becomes a constant, it makes sense to consider
the implicit constraint $\int_0^L u(t,x) {\rm d}x=0.$
This is successfully kept by the solution $u$ generated
from the solution $\theta$ of the SG equation.
Let $ x( \tau,s) $, $u( \tau,s) $ be the pair
obtained from the solution $ \theta( \tau , s)  $ 
of the SG equation under the boundary condition~\cref{eq_apb}. Then we see
\begin{align*}
\int^L_0 u(t,x) \rd x 
&= \int^S_0  u ( \tau, s ) x_s ( \tau ,s ) \rd s 
= \int^S_0 \left( u (\tau, 0 ) + \int^s_0 \sin \theta ( \tau , \sigma ) \rd \sigma \right) \cos \theta (\tau,s) \rd s \\
&=\int^S_0 \left( u (\tau, 0 ) + \int^s_0 \theta_{\tau s} ( \tau , \sigma ) \rd \sigma \right) \cos \theta (\tau,s) \rd s \\
&= \int^S_0 \left( u (\tau, 0 ) + \theta_{\tau} ( \tau , s ) - \theta_{\tau} (\tau , 0) \right) \cos \theta (\tau,s) \rd s \\
&= \left( u (\tau, 0 ) - \theta_{\tau} (\tau,0) \right) \int^S_0 \cos \theta ( \tau, s) \rd s + \frac{\rd}{\rd \tau} \int^S_0 \sin \theta (\tau,s) \rd s, 
\end{align*}
where the first and second term vanishes due to \cref{abp} and \cref{eq_proof_y}, respectively.

\section{Proposed method: a self-adaptive moving mesh scheme for the SP equation}
\label{sec_scheme}

We here propose the self-adaptive moving mesh scheme. 
First we show the outline in \cref{subsec_outline},
and its detail in the subsequent subsections.

\subsection{Outline}
\label{subsec_outline}

We give the outline of the proposed method for the SP equation.
Let us introduce the discrete variables $ \dthe{m}{k} $ approximating 
$ \theta ( m \Delta \tau , k \Delta s ) $ $(k \in \mathbb{Z} ; m = 0,1, \dots, M) $, 
where $ \Delta \tau $ and $ \Delta s $ are the mesh sizes. 
Here, in view of \cref{prop_pbc}, we impose the discrete boundary condition 
\begin{equation}
\dthe{m}{k+K} = \dthe{m}{k} + 2 n \pi \qquad ( k \in \mathbb{Z}; \ m = 0,1,\dots, M+1),
\label{eq_apbd}
\end{equation}
where $ n \in \mathbb{Z} $ is a fixed constant 
determined by the initial condition. 
We also introduce the vector
notation $ \dthe{m}{} := ( \dthe{m}{1} , \dots, \dthe{m}{K} )^{\top} $. 

The outline of the proposed method is defined as follows. 

\smallskip

\noindent 
{\bf Self-adaptive moving mesh method for the SP equation}
\begin{description}
\item[Step 0:] Prepare the initial condition $ \dthe{0}{k} \ (k \in \mathbb{Z} ) $ of the SG equation, and set $ n:= ( \dthe{0}{K} - \dthe{0}{0} )/ ( 2 \pi )  $ (\cref{subsec_ini}). 
\item[Step 1:] Compute $ \dthe{m}{k} \ (k=1,\dots,K; m = 1,\dots,M+1) $ by using~\cref{eq_sg_dvdm} under the boundary condition~\cref{eq_apbd} (\cref{subsec_dvdm}). 
\item[Step 2:] Compute $ \xd{m}{k} , \ud{m}{k} \ (k=1,\dots,K; m = 0, 1, \dots, M) $ by using \cref{eq_dhodom} with $ (\xd{m}{0}, \ud{m}{0} ) $ defined by \cref{eq_dhtm_t} (\cref{subsec_dht}). 
\end{description}

\subsection{Step 0: Preparation of the initial data $\dthe{0}{}$}
\label{subsec_ini}

We can think of two different cases depending on 
how the initial data for the SP equation is given.

The first possibility is the case that
the initial condition is directly given in
the parametrized form $( x (0,s) , u (0,s) ) $;
all multi-valued solutions, and some particular solutions
of interest such as~\cref{eq_spe_breather}, are given in this way.
In this case, we can prepare the initial data $\dthe{0}{}$ directly
in the computational domain.
We first choose an appropriate $S$,
the length of the computational
domain (possibly depending on the period of the solution).
Then $\dthe{0}{}$ can be computed by either of the following two ways.
If the initial condition is given analytically as functions,
the function $\theta(0,s)$ can be found easily by~\cref{iht} as
$\theta(0,s) = {\rm arctan} (u_s/x_s)$.
We can then evaluate the function on the grid points 
$s = k\Delta s$ ($k=1,\ldots,K$).
Cares are necessary to the argument range and on some points where $x_s=0$,
but this is straightforward by making the resulting function continuous.
If, on the other hand, $( x (0,s) , u (0,s) ) $ is given only numerically,
for example, only as a program subroutine,
and accordingly known only on the grid points,
we can replace the derivatives $u_s$ and $x_s$ above by 
some finite differences; for example, we can set (again with some care)
\begin{equation} \label{fd_approx}
\dthe{0}{k} = {\rm arctan} \left( \frac{\ud{0}{k} - \ud{0}{k-1}}{\xd{0}{k}-\xd{0}{k-1}} \right), \qquad (k=1,\ldots,K).
\end{equation}

The second possibility is that
the initial condition is given for the original SP equation~\cref{eq_spe}
as $u_0(x)$, which is 
more natural setting as an initial value problem for~\cref{eq_spe},
although it intrinsically allows only single-valued solutions.
In this case, a possible strategy to
prepare $\dthe{0}{}$ is as follows.
We first find the total arc-length $\mathcal{S}(u_0)$ by
an accurate numerical integration (recall~\cref{subsec_arc}),
and set $S = \mathcal{S}(u_0)$.
The remaining task is to find $\xd{0}{}$ and $\ud{0}{}$
on the grid points $s = k\Delta s$,
which reduces the discussion to the first case above.
An obvious way to do this is to numerically solve
for each $k$ ($k=1,\ldots,K$)
\[
\int_0^x \sqrt{1+(u_{0,\xi}(\xi))^2} {\rm d}{\xi} = k\Delta s
\]
by, for example, the Newton iteration,
to obtain $\xd{0}{k}$
(in above, $u_{0,\xi}$ means $\partial u_0/\partial \xi$).
 Then we can set $\ud{0}{k} = u_0(\xd{0}{k})$.
Another possible way is to employ a method to realize
the ``equidistribution principle''
in the standard moving mesh method~\cite[Chap.~2]{HR2011}.
In the setting of this paper, the distance is measured by
the arc-length (i.e., in terms of the mesh density function,
$\rho(x) = \sqrt{1+{u_x}^2}$, in the notation of~\cite[Section~2.4.3]{HR2011}).
Note that, however, in the standard moving mesh method,
often a rough realization of equidistribution is sufficient
for practical purposes, and the resulting grid may not be strictly
equidistributing, depending on the algorithm.
This can cause a severe accuracy degeneration in preparing $\dthe{0}{}$
in the present context, since we demand $s$ to be the strict 
arc-length.
In this sense, we should employ an accurate equidistribution method,
or alternatively,
a combination of a rough method and the Newton iteration refinement.

\subsection{Step 1: Solving the SG equation by a conservative method}
\label{subsec_dvdm}

As seen in \cref{sec_fixed}, 
the preservation of the functional $\mathcal{H}$
in the SG equation is indispensable
for keeping the physical window size constant.
The discrete variational derivative method~\cite{F1999}
(DVDM; see also Furihata--Matsuo~\cite{FM2011})
gives a way to construct such a special scheme.
Although the original DVDM did not cover the SG equation~\cref{sgevf},
where an additional differential operator with respect to the spatial variable $s$ is included in the left hand side,
later it has been extended to accommodate such 
cases~\cite{MYM2012,YMS2010}. 
Furthermore, recently, a scheme with better stability was
proposed in Furihata--Sato--Matsuo~\cite{FSM2016+}. 
Below we show a scheme for the SG equation employing their technique.

Let us first introduce the spatial forward difference and average operators 
\begin{align*}
\fd_s \dthe{m}{k} &= \frac{ \dthe{m}{k+1} - \dthe{m}{k} }{\Delta s}, & 
\fa_s \dthe{m}{k} &= \frac{ \dthe{m}{k+1} + \dthe{m}{k} }{2}, 
\end{align*}
and the spatial backward difference and average operators 
\begin{align*}
\bd_s \dthe{m}{k} &= \frac{ \dthe{m}{k} - \dthe{m}{k-1} }{\Delta s}, & 
\ba_s \dthe{m}{k} &= \frac{ \dthe{m}{k} + \dthe{m}{k-1} }{2}. 
\end{align*}
Here, the subscript $ s $ denotes the spatial operator, 
and the temporal counterparts $ \fd_{\tau} , \ \fa_{\tau}, \ \bd_{\tau} , \ \ba_{\tau} $ are similarly defined. 

First of all, we define a discrete version $ \mathcal{H}_{\rd} $ of the functional $ \mathcal{H} $ as
\begin{equation*}
\mathcal{H}_{\rd} (\dthe{m}{}) = - \sum_{k=1}^{K} \cos \dthe{m}{k} \Delta s.  
\end{equation*}
Then, the discrete variational derivative $ \delta \mathcal{H}_{\rd} / \delta ( \dthe{m+1}{}, \dthe{m}{} )_k $ of $ \mathcal{H}_{\rd} $ is defined as a function satisfying the relation
\begin{equation*}
\fd_{\tau} \mathcal{H}_{\rd} \left( \dthe{m}{} \right) = \sum_{k=1}^K \frac{ \delta \mathcal{H}_{\rd} }{\delta ( \dthe{m+1}{} , \dthe{m}{} )_k } \fd_{\tau} \dthe{m}{k} \Delta s, 
\end{equation*}
and in this case, it can be derived as 
\begin{equation*}
\frac{ \delta \mathcal{H}_{\rd} }{\delta ( \dthe{m+1}{} , \dthe{m}{} )_k } = - \frac{ \cos \dthe{m+1}{k} - \cos \dthe{m}{k} }{ \dthe{m+1}{k} - \dthe{m}{k} } \quad ( k = 1,2, \dots, K ).
\label{dvd_sge}
\end{equation*}
By using this, we define an average-difference scheme
\begin{equation}
\fd_s \fd_{\tau} \theta_k^{(m)} = \fa_s \frac{ \delta \mathcal{H}_{\rd} }{\delta ( \dthe{m+1}{} , \dthe{m}{} )_k }. 
\label{eq_sg_dvdm}
\end{equation}

\begin{proposition}[cf. \protect{\cite[Theorem~1]{FSM2016+}}]
Let $ \dthe{m}{k} $ be the solution of the discrete variational derivative method~\cref{eq_sg_dvdm} 
with a boundary condition satisfying 
\begin{equation*}
\left[ \left( \fd_{\tau} \dthe{m}{k} - \frac{\Delta s}{2} \frac{ \delta \mathcal{H}_{\rd} }{\delta ( \dthe{m+1}{} , \dthe{m}{} )_k} \right)^2 \right]^{K+1}_1 = 0.
\label{sgdvdmbcond}
\end{equation*}
Then, it holds that $ \mathcal{H}_{\rd} (\dthe{m+1}{}) = \mathcal{H}_{\rd} ( \dthe{m}{} )$ for any $ m = 0,\dots, M$.  
\label{propdvdm}
\end{proposition}

Since the proof of \cref{propdvdm} is rather complicated 
and the full proof has already appeared in \cite{FSM2016+}, 
we omit it here. 

\begin{remark} \label{rem_standard}
In discretizing \cref{sgevf}, one may be tempted to 
employ some standard skew-symmetric difference operators such as
the central difference operator, the compact difference operator (see, e.g.,~\cite{KMY2012}), and the 
Fourier-spectral difference operator~(see, e.g., \cite{Fornberg1996}),
since the skewness is basically the source of the conservation.
For example, 
by using the central difference operator $ \cd_s : \ud{m}{k} \mapsto ( \ud{m}{k+1} - \ud{m}{k-1} )/ 2 \Delta s$, 
which is the simplest skew-symmetric difference operator, 
we can also construct a conservative method 
\begin{equation}
\cd_s \fd_{\tau} \theta_k^{(m)} = \frac{ \delta \mathcal{H}_{\rd} }{\delta ( \dthe{m+1}{} , \dthe{m}{} )_k }. 
\label{eq_sg_dvdm_cd}
\end{equation}
Although the numerical method certainly conserves the value of $ \mathcal{H}_{\rd } $, it suffers from undesirable spatial oscillations (see, \cite{FSM2016+}).
\end{remark}

\subsection{Step 2: Obtaining the solution by a discrete hodograph transformation}
\label{subsec_dht}

The remaining task is to give a formula to transform $\dthe{m}{}$ 
back to $\xd{m}{}$ and $\ud{m}{}$.
If we employ the natural discretization 
\begin{equation}
\matbra{ \begin{array}{c} \xd{m}{k} \\[5pt] \ud{m}{k} \end{array} } = \matbra{ \begin{array}{c} \xd{m}{0} \\[5pt] \ud{m}{0} \end{array} } 
+ \sum_{i=1}^{k} \matbra{ \begin{array}{c} \cos \dthe{m}{i} \\[5pt] \sin \dthe{m}{i} \end{array} } \Delta s,
\label{eq_dht}
\end{equation}
of the hodograph transformation~\cref{ht}, 
the periodic boundary condition is not satisfied in the physical domain;
in fact,
$ \ud{m}{K} - \ud{m}{0} = \sum_{k=1}^K \sin \dthe{m}{k} \Delta s $ does not coincide with $0$. 
The reason lies in the fact that the discrete counterpart of~\cref{eq_proof_y} does not hold in general, 
due to the difference between the right-hand side of the discrete sine-Gordon equation~\cref{eq_sg_dvdm} and 
the $ \sin \dthe{m}{i} $ in~\cref{eq_dht}. 
For the case of periodic domain, this problem is serious, 
since it means $\ud{m}{K} \neq \ud{m}{0}$, and thus there appears a
gap in $u$ around the base point $\xd{m}{0}$.

We can overcome this difficulty  by carefully redefining the discrete hodograph transformation as
\begin{subequations}\label{eq_dhodom}
\begin{equation}\label{eq_dhtm}
\matbra{ \begin{array}{c} \xd{m}{k} \\[5pt] \ud{m}{k} \end{array} } = \matbra{ \begin{array}{c} \xd{m}{0} \\[5pt] \ud{m}{0} \end{array} } 
+ \sum_{i=1}^k \matbra{ \begin{array}{c} \cos \dthe{m}{i} \\[5pt] \ba_s \frac{\delta \mathcal{H}_{\rd} }{ \delta \left( \dthe{m+1}{} , \dthe{m}{} \right)_i } \end{array} } \Delta s,
\end{equation}
\begin{equation}\label{eq_dihtm}
\matbra{ \begin{array}{c} \bd_s \xd{m}{k} \\ \bd_s \ud{m}{k} \end{array} } = \matbra{ \begin{array}{c} \cos \dthe{m}{k} \\[5pt] \ba_s \frac{\delta \mathcal{H}_{\rd} }{ \delta \left( \dthe{m+1}{} , \dthe{m}{} \right)_k } \end{array} } . 
\end{equation}
\end{subequations}
To prove it works fine, we start by showing the following lemma,
which corresponds to \cref{lemL}.

\begin{lemma}
Suppose that $ ( \xd{m}{k} , \ud{m}{k} ) $ is obtained from $ \dthe{m}{k} $ by using~\cref{eq_dhtm}. 
Then, the relation $ \mathcal{L} \left( \xd{m}{} , \ud{m}{} \right) = - \mathcal{H} \left( \dthe{m}{} \right) $ holds for any $ m = 0,1,\dots, M $, 
where the functional $ \mathcal{L} $ is defined as $ \mathcal{L} ( \xd{m}{} , \ud{m}{} ) = \xd{m}{K} - \xd{m}{0} $. 
\label{lem_Ld}
\end{lemma}

\begin{proof}
It can be shown by the simple calculation: 
\[ \mathcal{L} \left( \xd{m}{} , \ud{m}{} \right) = \xd{m}{K} - \xd{m}{0} = \sum_{i=1}^K \cos \dthe{m}{i} = - \mathcal{H}_{\rd} \left( \dthe{m}{} \right) . \]
\end{proof}

The value of the functional $ \mathcal{L} ( \xd{m}{} , \ud{m}{} ) $ represents 
the window size in the physical space. 
Thanks to the \cref{lem_Ld} and the consistent definition~\cref{eq_dhtm}, 
we can follow the line of the discussion in the proof of \cref{prop_pbc}.

\begin{theorem}\label{thm_pbd}
Let $ \dthe{m}{k} \ (k \in \mathbb{Z} ;\  m=0,\dots,M+1)$ be the solution of the discrete sine-Gordon equation~\cref{eq_sg_dvdm} under the boundary condition~\cref{eq_apbd} for some $ n \in \mathbb{Z} $. 
If $ ( \xd{m}{k} , \ud{m}{k} ) $ is the curve obtained from $ \dthe{m}{k} $ 
via the discrete hodograph transformation~\cref{eq_dhtm} with some $ ( \xd{m}{0} , \ud{m}{0} )$, 
then $ ( \xd{m}{k} , \ud{m}{k} ) $ satisfies the boundary condition 
\begin{equation}
\xd{m}{k+K} = \xd{m}{k} + L, \quad \ud{m}{k+K} = \ud{m}{k} \qquad (k \in \mathbb{Z} ; \  m = 0, \dots, M+1), 
\label{eq_pbd}
\end{equation}
where the constant $ L $ is determined by the initial condition $ L := - \mathcal{H}_{\rd} ( \dthe{0}{} ) $. 
\end{theorem}

\begin{proof}
From the boundary condition~\cref{eq_apbd} and the definition~\cref{eq_dhtm} of the discrete hodograph transformation, we obtain 
\begin{align*}
\matbra{ \begin{array}{c} \xd{m}{k+K} \\[5pt] \ud{m}{k+K} \end{array} } - \matbra{ \begin{array}{c} \xd{m}{k} \\[5pt] \ud{m}{k} \end{array} } 
&= \sum_{i = k +1}^{k+K} \matbra{ \begin{array}{c} \cos \dthe{m}{i} \\[5pt] \ba_s \frac{\delta \mathcal{H}_{\rd} }{ \delta \left( \dthe{m+1}{} , \dthe{m}{} \right)_i } \end{array} } \Delta s
= \sum_{i =1}^{K} \matbra{ \begin{array}{c} \cos \dthe{m}{i} \\[5pt] \ba_s \frac{\delta \mathcal{H}_{\rd} }{ \delta \left( \dthe{m+1}{} , \dthe{m}{} \right)_i } \end{array} } \Delta s  . 
\end{align*}
Hence, the boundary condition~\cref{eq_pbd} is equivalent to the equation
\begin{equation*}
\sum_{i =1}^{K} \matbra{ \begin{array}{c} \cos \dthe{m}{i} \\[5pt] \ba_s \frac{\delta \mathcal{H}_{\rd} }{ \delta \left( \dthe{m+1}{} , \dthe{m}{} \right)_i } \end{array} } \Delta s = \matbra{ \begin{array}{c} L \\ 0 \end{array} }. 
\end{equation*}
The equality above can be verified by following the line of the discussion in the proof of \cref{prop_pbc} 
as follows: 
\begin{equation*}
\sum_{i =1}^{K} \cos \dthe{m}{i} \Delta s = - \mathcal{H}_{\rd} \left( \dthe{m}{} \right) = - \mathcal{H} \left(  \dthe{0}{} \right), \\ 
\end{equation*}
\begin{equation*}
\sum_{i=1}^{K} \ba_s \frac{\delta \mathcal{H}_{\rd} }{ \delta \left( \dthe{m+1}{} , \dthe{m}{} \right)_i } \Delta s = 
\sum_{i=0}^{K-1} \fa_s \frac{\delta \mathcal{H}_{\rd} }{ \delta \left( \dthe{m+1}{} , \dthe{m}{} \right)_i } \Delta s 
= \sum_{i=0}^{K-1} \fd_s \fd_{\tau} \dthe{m}{i} \Delta s = 0. \label{proof_byd}
\end{equation*}
Here, we used \cref{propdvdm} and the boundary condition~\cref{eq_apbd}. 
Note that the first identity exactly corresponds to~\cref{eq_cos_H}.
The second is formally different from~\cref{eq_proof_y},
although the correpondence holds in the continuous limit.
\end{proof}

Our final task is to give 
a method to compute $ ( \xd{m}{0} , \ud{m}{0} ) $. 
We can do this by defining a discrete counterpart 
of the condition~\cref{abp}, which describes the time evolution
of the base point.
Fortunately, \cref{thm_pbd} reveals that regardless of this choice
the discrete boundary condition is guaranteed.
Let us here choose a discrete counterpart  as follows: 
\begin{subequations}\label{eq_dhtm_t}
\begin{align}
\xd{m+1}{0} &= \xd{m}{0} -  \frac{ \Delta \tau }{2} \left( \ud{m}{0} \right)^2 , \label{eq_dhtm_tx} \\
\ud{m}{0} &= \fd_{\tau} \dthe{m}{0} - \frac{ \sum_{k=1}^K \left( \fd_{\tau} \dthe{m}{k} \right) \cos \dthe{m}{k} \Delta s }{ \sum_{k=1}^K \cos \dthe{m}{k} \Delta s }. \label{eq_dhtm_tu} 
\end{align}
\end{subequations}
For~\cref{eq_dhtm_tu},
one might think that a simpler definition 
$\ud{m}{0} = \fd_{\tau} \dthe{m}{0}$ would suffice,
and~\cref{eq_dhtm_tu} is an unnecessary complication.
But it is necessary to replicate
the implicit constraint $ \int^L_0 u(t,x) \rd x = 0 $  of the SP equation. 
Recall the discussion in~\cref{sec_implicit}; we here copy it
in the discrete setting.
For the numerical solution, $ \xd{m}{k} , \ud{m}{k} $, we see that 
(with an abbreviation $ a_k :=  \delta \mathcal{H}_{\rd} / \delta ( \dthe{m+1}{}, \dthe{m}{} )_k  $)
\begin{align*}
\sum_{k=1}^K \ud{m}{k} \left( \xd{m}{k} - \xd{m}{k-1} \right)
&= \sum_{k=1}^K \left( \ud{m}{0} + \sum_{i=1}^k \ba_s a_i \Delta s \right) \cos \dthe{m}{k} \Delta s \\
&= \sum_{k=1}^K \left( \ud{m}{0} + \sum_{i=0}^{k-1} \fa_s a_i \Delta s \right) \cos \dthe{m}{k} \Delta s \\
&= \sum_{k=1}^K \left( \ud{m}{0} + \sum_{i=0}^{k-1} \fd_s \fd_{\tau} \dthe{m}{i} \Delta s \right) \cos \dthe{m}{k} \Delta s \\
&= \sum_{k=1}^K \left( \ud{m}{0} + \fd_{\tau} \dthe{m}{k} - \fd_{\tau} \dthe{m}{0} \right) \cos \dthe{m}{k} \Delta s \\
&= \left( \ud{m}{0} - \fd_{\tau} \dthe{m}{0} \right) \sum_{k=1}^K \cos \dthe{m}{k} \Delta s + \sum_{k=1}^K \cos \dthe{m}{k}  \fd_{\tau} \dthe{m}{k} \Delta s.
\end{align*}
It successfully vanishes under \cref{eq_dhtm_tu}.
This is the desired discrete constraint.

Note that \cref{eq_dhtm_tx} (or its continuous counterpart~\cref{abp})
implies that the base point moves along time.
This might seem strange at first sight, but this is necessary;
the whole framework becomes consistent only under this situation.

\begin{remark} \label{rem:R}
Let us here summarize the intrinsic difference between
the $\R$ and the periodic cases.
The vanishing condition on $\R$ can be approximated by 
setting a sufficiently large window $[0,L]$, and by imposing
Dirichlet boundary condition
 $ u ( t, 0 ) = u (t,L) = u_x (t,0) = u_x (t,L) = 0$
(actually this is often employed in the integrable systems studies).
The corresponding boundary condition for the SG equation is
simply the Dirichlet boundary condition $ \theta ( \tau, 0 ) = \theta ( \tau, S) = 0 $ (or $= 2n\pi$ for loop solitons). 
This greatly simplifies the discussion carefully carried out
in the present paper; 
the condition~\cref{abp} is trivially satisfied
with $x(\tau,0)=0$, $u(\tau,0)=0$ (i.e., in this case
\emph{the base point does not move}),
and accordingly, we  do not need \cref{eq_dhtm_t}.
We also note that the assumption on the discrete boundary condition of the discrete conservation law (\cref{propdvdm}) holds, 
and the counterpart of \cref{thm_pbd} can be similarly proved. 
In this way, the results in this paper trivially applies
also to the $\R$ case.

Another note should go to the fact that,
if we employ the simple discrete hodograph transformation~\cref{eq_dht} instead of \cref{eq_dhtm}, 
we can prove the boundedness of the solution curve~\cite{OgumaMT}. 
On the other hand, the discrete boundary condition corresponding to $ u (t,L ) = 0 $ no longer exactly holds, i.e., the counterpart of \cref{thm_pbd} ceases to work. 
\end{remark}

\section{Numerical examples}
\label{sec_ne}

We here also introduce two standard schemes
to compare with the proposed scheme.
A norm-preserving scheme and
the multi-symplectic scheme~\cite{PK2015}.
All the schemes are nonlinearly implicit, but
through preliminary numerical tests we concluded that 
explicit methods are not adequate for the SP equation, which is
tougher than usual PDEs.
We solved nonlinear equations simply by `fsolve' of MATLAB 2013R;
in this paper we do not step into fast implementation of nonlinear solvers,
and leave it to future works.

\subsection{Structure-preserving schemes on fixed uniform mesh}
\label{subsec_sp_normal}

First, we show a norm-preserving scheme for the SP equation on the fixed uniform mesh, 
which has been already appeared in Introduction (\cref{fig_SPDVDM}). 

On the fixed mesh, the symbol $ \bud{m}{k} $ denotes the approximation of $ u(m \Delta t , k \Delta x )  $ for $m=0,1,\dots,M $ and $ k \in \mathbb{Z}$, 
and we assume the periodic boundary condition $ \bud{m}{k+N} = \bud{m}{k} $ ($N$ is a fixed integer satisfying $ N \Delta x = L $).  
As a discrete counterpart of the norm~$ \mathcal{I} $, we define 
\begin{equation*}
\mathcal{I}_{\rd} ( \bud{m}{k} ) := \frac{1}{2} \sum_{k=1}^N \left( \bud{m}{k} \right)^2 \Delta x. 
\end{equation*}
Then, one of the discrete variational derivative can be defined as
\begin{equation*}
\frac{\delta \mathcal{I}_{\rd}}{\delta \left( \bud{m+1}{}, \bud{m}{} \right)_k} := \fa_t \bud{m}{k}.
\end{equation*}
In order to maintain the skew-symmetry of the operator, 
the norm-preserving scheme can be constructed as 
\begin{equation}
\fd_t \bud{m}{k} = \left( \acd + \left( \cd_x \fa_t \bud{m}{k} \right) \acd \left( \cd_x \fa_t \bud{m}{k} \right) \right) \frac{\delta \mathcal{I}_{\rd}}{\delta \left( \bud{m+1}{}, \bud{m}{} \right)_k}, 
\label{eq_std_dvdm}
\end{equation}
where the discrete counterpart $ \acd $ of antiderivative $ \partial_x^{-1} $ is defined as 
\begin{align*}
\acd \bud{m}{k} &:= \tacd \bud{m}{k} - \frac{1}{N \Delta x } \sum_{\ell=1}^{N} \tacd \bud{m}{\ell} \Delta x, \\
\tacd \bud{m}{k} &:= \begin{cases} \frac{\bud{m}{N} + \bud{m}{1}}{2} \Delta x \quad & ( k = 1), \\ \left( \frac{1}{2} \bud{m}{N} + \sum_{\ell= 1}^{k-1} \bud{m}{\ell} \Delta x + \frac{1}{2} \bud{m}{k} \right) \Delta x \quad & (\text{otherwise}) \end{cases}
\end{align*}
(it is devised by Yaguchi--Matsuo--Sugihara~\cite{YMS2010} for the Ostrovsky equation). 
By using the relation $ \fd_x \acd = \acd \fd_x = \fa_x $, which can be proved in a straight forward manner, 
the numerical scheme~\cref{eq_std_dvdm} can be transformed into 
\begin{equation}
\fd_x \fd_t \bud{m}{k} = \fa_x \fa_t \bud{m}{k} + \fd_x \left( \left( \cd_x \fa_t \bud{m}{k} \right) \fa_x \frac{ \left( \fa_t \bud{m}{k} \right) \left( \fa_t \bud{m}{k-1} \right) }{2} \right)
\label{eq_std_dvdm2}
\end{equation}

In order to prove the discrete conservation law, 
the skew-symmetry of the discrete antiderivative $ \acd $ is important. 

\begin{lemma}[\protect{\cite[Lemma~6]{YMS2010}}]
For any $N$-periodic sequences $v,w$ such that satisfy $ \sum_k v_k =\sum_k w_k =0 $, it holds that 
$ \sum_{k=1}^N v_k \acd w_k  + \sum_{k=1}^N w_k \acd v_k  = 0$. 
In particular, $ \sum_{k=1}^{N} v_k \acd v_k  = 0 $
holds. 
\end{lemma}

By using the skew-symmetry of $ \acd $ above, 
we can prove the discrete conservation law as follows: 

\begin{proposition}
Let $ \bud{m}{k} $ be the numerical solution of the scheme~\cref{eq_std_dvdm} under the 
periodic boundary condition and initial condition satisfying $ \sum_k \bud{0}{k} = 0$. 
Then, $ \mathcal{I}_{\rd} ( \bud{m+1}{} ) = \mathcal{I}_{\rd} ( \bud{m}{} ) $ holds. 
\end{proposition}

\begin{proof}
We introduce the notation $ a_k = \delta \mathcal{I}_{\rd} / \delta ( \bud{m+1}{} , \bud{m}{} )_k $ for brevity. 
From the definition of the discrete variational derivative, we see 
\begin{align*}
\fd_t \mathcal{I}_{\rd} \left( \bud{m}{} \right)
&= \sum_{k=1}^N a_k \fd_t \bud{m}{k} \Delta x \\
&= \sum_{k=1}^N a_k \left( \acd + \left( \cd_x \fa_t \bud{m}{k} \right) \acd \left( \cd_x \fa_t \bud{m}{k} \right) \right) a_k \Delta x.
\end{align*}
Hence, in order to prove the proposition, it is sufficient to prove 
(i) $ \sum_k a_k = 0 $ and (ii) $ \sum_k a_k \cd_x \fa_t \bud{m}{k} =0 $. 

The condition (i) can be shown by proving the following claim: 
if $ \sum_k \bud{m}{k} = 0 $ holds, then $ \sum_k \bud{m+1}{k} = 0 $ holds. 
Since the numerical solution $ \bud{m}{k} $ of the numerical scheme~\cref{eq_std_dvdm} 
satisfies~\cref{eq_std_dvdm2}, by summing up the both sides of~\cref{eq_std_dvdm2} for $k =1,\dots, N$, we see
$
\sum_{k=1}^N \bud{m+1}{k} = -\sum_{k=1}^N \bud{m}{k}, 
$
which proves the claim. 

Finally, we prove the condition (ii): 
\begin{align*}
\sum_{k=1}^N a_k \cd_x \fa_t \bud{m}{k} &= \sum_{k=1}^N \left( \fa_t \bud{m}{k} \right) \cd_x \left( \fa_t \bud{m}{k} \right) = 0,
\end{align*}
where the last equality holds by the skew-symmetry of the central difference operator~$ \cd_x $. 
\end{proof}


Next we show the multi-symplectic scheme given in Pietrzyk--Kanatt\v{s}ikov~\cite{PK2015}.
They found a multi-symplectic form of the SP equation, and applied 
the Preissman-box scheme~\cite{BR2001} to it.
By eliminating redundant variables in Pietrzyk--Kanatt\v{s}ikov's 
scheme, we find
\[ \fd_x \fd_t \left( \fa_x \right)^2 \bud{m}{k} = \fa_t \left( \fa_x \right)^3 \bud{m}{k} + \frac{1}{6} \left( \fd_x \right)^2 \left( \fa_t \fa_x \bud{m}{k} \right)^3. \]
By further rewriting $ \bud{m}{k} = \fa_x \bud{m}{k} $, 
the scheme above can be simplified into 
\begin{equation} \label{scheme_ms}
\cd_x \fd_t \bud{m}{k}  = \fa_t \caa_x \bud{m}{k} + \frac{1}{6} \cdd_x \left( \fa_t \bud{m}{k} \right)^3,
\end{equation}
where $ \cdd_x \bud{m}{k} := ( \bud{m}{k+1} - 2 \bud{m}{k} + \bud{m}{k}  )/ ( \Delta x )^2  $ and 
$ \caa_x \bud{m}{k} := ( \bud{m}{k+1} + 2 \bud{m}{k} + \bud{m}{k}  )/ 4 $. 
It should be noted that the solutions of this numerical scheme almost preserve $ \mathcal{I} $ and $ \mathcal{E} $
~\cite{MR2003}. 
For the general theory of  multi-symplectic integrators, see
Bridges--Reich~\cite{BR2001}.

\subsection{Numerical examples for smooth pulse solution}
\label{subsec_ne_breather}
 
First of all, we consider the pulse
solution~\cref{eq_spe_breather} with $ \xi = 0.3$.
\cref{fig_pulse30_SPDVDM_511,fig_pulse30_SPMS_511,fig_pulse30_SPSGDVDM_511,fig_pulse30_norm} show the results of
the norm-preserving scheme~\cref{eq_std_dvdm},
the multi-symplectic scheme~\cref{scheme_ms},
and the proposed method.
The parameters are taken to
$ K = 511, \ \Delta t = 0.01 , \ S = 70, \ L = 66.96$.
The initial data $\dthe{0}{}$ is prepared by
the finite difference approximation~\cref{fd_approx},
which seems fine judging from the results below.
The norm-preserving scheme is unstable, as seen in Introduction.
This shows that being structure-preserving in some sense is
not necessarily sufficient for difficult problems.
The other two schemes work fine for the setting.
This observation is consistent with Pietrzyk--Kanatt\v{s}ikov~\cite{PK2015}, 
which concluded that the multi-symplectic scheme is better than the standard method in case of $ \xi = 0.2 $ in view of the accuracy. 
The results for severer case $ \xi = 0.38 $
are shown in 
\cref{fig_SPDVDM_511,fig_SPMS_511,fig_SPSGDVDM_511,fig_pulse38_norm} ($ K = 511, \ \Delta t = 0.01 , \  S = 70, \ L = 67.6$).
The multi-symplectic scheme cannot withstand anymore,
and exhibit oscillations.
The proposed method keeps working, which shows the superiority of our method.
In fact, the method happily works with 
much coarser spatial and temporal meshes
for a longer time interval; see~\cref{fig_pulse38_long_SPSGDVDM}.
Note that although the proposed method does not directly aim at
the pulse (solitary wave) solution, as opposed to the integrable 
systems studies, it successfully captures the pulse
solution even at $t=1000$ (the pulse goes round the spatial interval
more than 10 times.)
\cref{fig_pulse30_norm} and \cref{fig_pulse38_norm} show that the
preservation of the invariants is consistent with the
construction of the schemes.
There, ``NP'' means the norm-preserving scheme,
``MS'' the multi-symplectic scheme,
and ``SG'' the proposed scheme.

\begin{figure}[htp]
\begin{minipage}{0.48\textwidth}
\centering
\pgfplotsset{width=6cm,compat=newest}
\begin{tikzpicture}
\begin{axis}[
		xlabel=$x$, 
		ylabel=$t$, 
		zlabel=$u$, 
		view={30}{45},
		enlarge x limits=false]
\addplot3[black] table {stdDVDM30pulse0.dat};
\addplot3[black] table {stdDVDM30pulse2.dat};
\addplot3[black] table {stdDVDM30pulse4.dat};
\addplot3[black] table {stdDVDM30pulse6.dat};
\addplot3[black] table {stdDVDM30pulse8.dat};
\addplot3[black] table {stdDVDM30pulse10.dat};
\addplot3[black] table {stdDVDM30pulse12.dat};
\addplot3[black] table {stdDVDM30pulse14.dat};
\addplot3[black] table {stdDVDM30pulse16.dat};
\addplot3[black] table {stdDVDM30pulse18.dat};
\addplot3[black] table {stdDVDM30pulse20.dat};
\end{axis}
\end{tikzpicture}
\caption{The numerical solution of the norm-preserving scheme~\cref{eq_std_dvdm} corresponding to the pulse solution~\cref{eq_spe_breather} with $ \xi = 0.30 $ $( K = 511 $, $ \Delta t = 0.01 )$.}
\label{fig_pulse30_SPDVDM_511}
\end{minipage}
\hspace{3pt}
\begin{minipage}{0.48\textwidth}
\centering
\pgfplotsset{width=6cm,compat=newest}
\begin{tikzpicture}
\begin{axis}[
		xlabel=$x$, 
		ylabel=$t$, 
		zlabel=$u$, 
		view={30}{45},
		enlarge x limits=false]
\addplot3[black] table {stdMS30pulse0.dat};
\addplot3[black] table {stdMS30pulse2.dat};
\addplot3[black] table {stdMS30pulse4.dat};
\addplot3[black] table {stdMS30pulse6.dat};
\addplot3[black] table {stdMS30pulse8.dat};
\addplot3[black] table {stdMS30pulse10.dat};
\addplot3[black] table {stdMS30pulse12.dat};
\addplot3[black] table {stdMS30pulse14.dat};
\addplot3[black] table {stdMS30pulse16.dat};
\addplot3[black] table {stdMS30pulse18.dat};
\addplot3[black] table {stdMS30pulse20.dat};
\end{axis}
\end{tikzpicture}
\caption{The numerical solution of the multi-symplectic scheme corresponding to the pulse solution~\cref{eq_spe_breather} with $ \xi = 0.30 $ $( K = 511 $, $ \Delta t = 0.01 )$.}
\label{fig_pulse30_SPMS_511}
\end{minipage}

\begin{minipage}{0.48\textwidth}
\centering
\pgfplotsset{width=6cm,compat=newest}
\begin{tikzpicture}
\begin{axis}[
		xlabel=$x$, 
		ylabel=$t$, 
		zlabel=$u$, 
		view={30}{45},
		enlarge x limits=false]
\addplot3[black] table {SGDVDM30pulse0.dat};
\addplot3[black] table {SGDVDM30pulse2.dat};
\addplot3[black] table {SGDVDM30pulse4.dat};
\addplot3[black] table {SGDVDM30pulse6.dat};
\addplot3[black] table {SGDVDM30pulse8.dat};
\addplot3[black] table {SGDVDM30pulse10.dat};
\addplot3[black] table {SGDVDM30pulse12.dat};
\addplot3[black] table {SGDVDM30pulse14.dat};
\addplot3[black] table {SGDVDM30pulse16.dat};
\addplot3[black] table {SGDVDM30pulse18.dat};
\addplot3[black] table {SGDVDM30pulse20.dat};
\end{axis}
\end{tikzpicture}
\caption{The numerical solution of the proposed method under the initial condition~\cref{eq_spe_breather} with $ \xi = 0.30 $ $( K = 511 $, $ \Delta t = 0.01 )$.}
\label{fig_pulse30_SPSGDVDM_511}
\end{minipage}
\begin{minipage}{0.48\textwidth}
\begin{tikzpicture}
\begin{axis}[width=5.5cm,compat = newest,xlabel={$t$},enlarge x limits=false, 
legend style={at={(0.5,-0.15)},anchor=north,legend columns=2},
ylabel ={$ \log $(relative error)}]
\addplot[smooth,green] table {stdDVDM30pulseLogenerrordown.dat};
\addlegendentry{\footnotesize NP: energy}
\addplot[densely dotted,green] table {stdDVDM30pulseLognormerrordown.dat};
\addlegendentry{\footnotesize NP: norm}
\addplot[smooth,blue] table {stdMS30pulseLogenerrordown.dat};
\addlegendentry{\footnotesize MS: energy}
\addplot[densely dotted,blue] table {stdMS30pulseLognormerrordown.dat};
\addlegendentry{\footnotesize MS: norm}
\addplot[smooth,red] table {SGDVDM30pulseLogenerrordown.dat};
\addlegendentry{\footnotesize SG: energy}
\addplot[densely dotted,red] table {SGDVDM30pulseLognormerrordown.dat};
\addlegendentry{\footnotesize SG: norm}
\end{axis}
\end{tikzpicture}
\caption{The evolution of the logarithm of the relative errors of the norm $ \mathcal{I} $ and the energy $ \mathcal{E} $ for each numerical solutions.}
\label{fig_pulse30_norm}
\end{minipage}
\end{figure}

\begin{figure}[htp]
\begin{minipage}{0.48\textwidth}
\centering
\pgfplotsset{width=6cm,compat=newest}
\begin{tikzpicture}
\begin{axis}[
		xlabel=$x$, 
		ylabel=$t$, 
		zlabel=$u$, 
		view={10}{30},
		enlarge x limits=false]
\addplot3[black] table {stdDVDM38pulse0.dat};
\addplot3[black] table {stdDVDM38pulse2.dat};
\addplot3[black] table {stdDVDM38pulse4.dat};
\addplot3[black] table {stdDVDM38pulse6.dat};
\addplot3[black] table {stdDVDM38pulse8.dat};
\addplot3[black] table {stdDVDM38pulse10.dat};
\addplot3[black] table {stdDVDM38pulse12.dat};
\addplot3[black] table {stdDVDM38pulse14.dat};
\addplot3[black] table {stdDVDM38pulse16.dat};
\addplot3[black] table {stdDVDM38pulse18.dat};
\addplot3[black] table {stdDVDM38pulse20.dat};
\end{axis}
\end{tikzpicture}
\caption{The numerical solution of the norm-preserving scheme~\cref{eq_std_dvdm} corresponding to the pulse solution~\cref{eq_spe_breather} with $ \xi = 0.38 $ $( K = 511 $, $ \Delta t = 0.01 )$.}
\label{fig_SPDVDM_511}
\end{minipage}
\hspace{3pt}
\begin{minipage}{0.48\textwidth}
\centering
\pgfplotsset{width=6cm,compat=newest}
\begin{tikzpicture}
\begin{axis}[
		xlabel=$x$, 
		ylabel=$t$, 
		zlabel=$u$, 
		view={10}{30},
		enlarge x limits=false]
\addplot3[black] table {stdMS38pulse0.dat};
\addplot3[black] table {stdMS38pulse2.dat};
\addplot3[black] table {stdMS38pulse4.dat};
\addplot3[black] table {stdMS38pulse6.dat};
\addplot3[black] table {stdMS38pulse8.dat};
\addplot3[black] table {stdMS38pulse10.dat};
\addplot3[black] table {stdMS38pulse12.dat};
\addplot3[black] table {stdMS38pulse14.dat};
\addplot3[black] table {stdMS38pulse16.dat};
\addplot3[black] table {stdMS38pulse18.dat};
\addplot3[black] table {stdMS38pulse20.dat};
\end{axis}
\end{tikzpicture}
\caption{The numerical solution of the multi-symplectic scheme corresponding to the pulse solution~\cref{eq_spe_breather} with $ \xi = 0.38 $ $( K = 511 $, $ \Delta t = 0.01 )$.}
\label{fig_SPMS_511}
\end{minipage}

\begin{minipage}{0.48\textwidth}
\centering
\pgfplotsset{width=6cm,compat=newest}
\begin{tikzpicture}
\begin{axis}[
		xlabel=$x$, 
		ylabel=$t$, 
		zlabel=$u$, 
		view={10}{30},
		enlarge x limits=false]
\addplot3[black] table {SGDVDM38pulse0.dat};
\addplot3[black] table {SGDVDM38pulse2.dat};
\addplot3[black] table {SGDVDM38pulse4.dat};
\addplot3[black] table {SGDVDM38pulse6.dat};
\addplot3[black] table {SGDVDM38pulse8.dat};
\addplot3[black] table {SGDVDM38pulse10.dat};
\addplot3[black] table {SGDVDM38pulse12.dat};
\addplot3[black] table {SGDVDM38pulse14.dat};
\addplot3[black] table {SGDVDM38pulse16.dat};
\addplot3[black] table {SGDVDM38pulse18.dat};
\addplot3[black] table {SGDVDM38pulse20.dat};
\end{axis}
\end{tikzpicture}
\caption{The numerical solution of the proposed method under the initial condition~\cref{eq_spe_breather} with $ \xi = 0.38 $ $( K = 511 $, $ \Delta t = 0.01 )$.}
\label{fig_SPSGDVDM_511}
\end{minipage}
\begin{minipage}{0.48\textwidth}
\begin{tikzpicture}
\begin{axis}[width=5.5cm,compat = newest,xlabel={$t$},enlarge x limits=false, 
legend style={at={(0.5,-0.15)},anchor=north,legend columns=3},
ylabel ={$ \log $(error of the norm)}]
\addplot[black,densely dashed] table {stdDVDM38pulseLogenerror.dat};
\addlegendentry{NP}
\addplot[black,smooth] table {stdMS38pulseLogenerror.dat};
\addlegendentry{MS}
\addplot[black,densely dotted] table {SGDVDM38pulseLogenerror.dat};
\addlegendentry{SG}
\end{axis}
\end{tikzpicture}
\caption{The evolution of the logarithm of the relative error of the norm $ \mathcal{I}_{\rd} $ for each numerical solutions.}
\label{fig_pulse38_norm}
\end{minipage}
\end{figure}

\begin{figure}
\begin{minipage}{0.48\textwidth}
\centering
\pgfplotsset{width=6cm,compat=newest}
\begin{tikzpicture}
\begin{axis}[
		xlabel=$x$, 
		ylabel=$t$, 
		zlabel=$u$, 
		view={30}{45},
		enlarge x limits=false]
\addplot3[black] table {pulse38spSPSGDVDM0.dat};
\addplot3[black] table {pulse38spSPSGDVDM1.dat};
\addplot3[black] table {pulse38spSPSGDVDM2.dat};
\addplot3[black] table {pulse38spSPSGDVDM3.dat};
\addplot3[black] table {pulse38spSPSGDVDM4.dat};
\addplot3[black] table {pulse38spSPSGDVDM5.dat};
\addplot3[black] table {pulse38spSPSGDVDM6.dat};
\addplot3[black] table {pulse38spSPSGDVDM7.dat};
\addplot3[black] table {pulse38spSPSGDVDM8.dat};
\addplot3[black] table {pulse38spSPSGDVDM9.dat};
\addplot3[black] table {pulse38spSPSGDVDM10.dat};
\addplot3[black] table {pulse38spSPSGDVDM11.dat};
\addplot3[black] table {pulse38spSPSGDVDM12.dat};
\addplot3[black] table {pulse38spSPSGDVDM13.dat};
\addplot3[black] table {pulse38spSPSGDVDM14.dat};
\addplot3[black] table {pulse38spSPSGDVDM15.dat};
\addplot3[black] table {pulse38spSPSGDVDM16.dat};
\addplot3[black] table {pulse38spSPSGDVDM17.dat};
\addplot3[black] table {pulse38spSPSGDVDM18.dat};
\addplot3[black] table {pulse38spSPSGDVDM19.dat};
\addplot3[black] table {pulse38spSPSGDVDM20.dat};
\end{axis}
\end{tikzpicture}
\end{minipage}
\begin{minipage}{0.48\textwidth}
\pgfplotsset{width=5.5cm}
\centering
\begin{tikzpicture}
\begin{axis}[compat = newest,xlabel={$x$},ylabel={$u$},enlarge x limits=false]
\addplot[black,smooth] table {pulse38spSPSGDVDMt1000.dat};
\end{axis}
\end{tikzpicture}
\end{minipage}
\caption{The numerical solution of the proposed method corresponding to the pulse solution~\cref{eq_spe_breather} with $ \xi = 0.38 $ $( K = 117 $, $ \Delta t = 0.1 )${\em :} the left panel shows the propagation of the numerical solution until $ t = 60$, the right panel shows the numerical solution at $ t = 1000 $.}
\label{fig_pulse38_long_SPSGDVDM}
\end{figure}
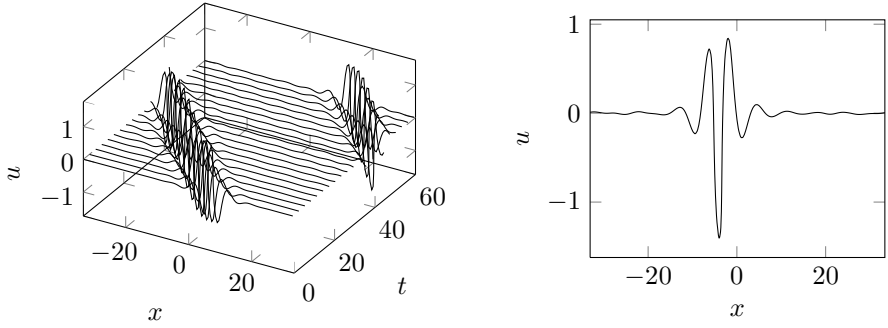

\subsection{Numerical examples for the loop and anti-loop solution}
\label{ne_circle_2loop}

In what follows, let us consider exotic solutions,
which are more difficult to treat numerically.
Recall in \cref{rem_standard}
a standard DVDM scheme~\cref{eq_sg_dvdm_cd} 
using the central difference operator was introduced.
In our preliminary numerical test, we confirmed that it 
beautifully reproduced the smooth pulse solution shown
in the previous section (omitted here.)
However, it ceases to work for more difficult cases,
and in such a situation, the use of
the average-difference method~\cref{eq_sg_dvdm} is indispensable.
Let us consider the following loop and anti-loop solution
\begin{equation}
\begin{cases}
{\displaystyle u ( \tau,s ) = 4 \xi \zeta \frac{ \xi \sinh \psi \sinh \phi  + \zeta \cosh \psi \cosh \phi }{ \xi^2 \sinh^2 \psi + \zeta^2 \cosh^2 \phi }, }\\[10pt]
{\displaystyle x(\tau ,s) = s + 2 \xi \zeta \frac{ \xi \sinh 2 \psi - \zeta \sinh 2 \phi }{ \xi^2 \sinh^2 \psi + \zeta^2 \cosh^2 \phi } },
\end{cases}
\label{eq_lals_spe}
\end{equation}
where the parameter $ \xi > 1$ and 
\[ \phi  = \xi ( s + \tau ) , \quad \psi = \zeta (s-\tau) , \quad \zeta = \sqrt{\xi^2 - 1}. \]
Again, in what follows, we employed the finite difference approximation~\cref{fd_approx} for preparing the initial data.

As shown in \cref{fig_per_lal_spe,fig_per_lal_sge}, 
the standard DVDM~\cref{eq_sg_dvdm} cannot reproduce the loop and anti-loop solution with $ \xi = 1.2 $
under the condition $ \Delta t = 0.1$, $ S = 80 $, and $ K = 257 $. 
On the other hand, even if we employ larger mesh sizes $ K = 129$, 
the proposed method reproduce the solution until $ t = 800 $ as shown in 
\cref{fig_per_lal_ADspe,fig_per_lal_ADsge}. 
The bars in \cref{fig_per_lal_ADspe} denote the grids;
we can observe that the grids become dense around the loop and 
anti-loop solutions, which reflects the moving mesh property
of the proposed method.
The bold circle ``{\large $\bullet$}''
 denotes the base point (recall~\cref{abp}
for the continuous case, and \cref{eq_dhtm_tx} and \cref{eq_dhtm_tu}
for the discrete case.)
It is observed that the base point moves as time goes by
(which is demanded in the periodic case),
and furthermore, the solution is successfully continuous around
that point.
Those notes apply also to other figures with similar bars and circles.

\begin{figure}[htbp]
\begin{minipage}{\textwidth}
\pgfplotsset{width=4cm}
\centering
\begin{minipage}{0.32\textwidth}
\centering
\begin{tikzpicture}[remember picture]
\begin{axis}[compat = newest,xlabel={$x$},ylabel={$u$},enlarge x limits=false]
\addplot[black,smooth] table[x=x,y=y] {perlalp0xy.dat};
\addplot[mark = *,only marks,black] coordinates
{(70.37,-0.00)};
\addplot[black,only marks,mark = |] table[x=x,y=c] {perlalp0xyd.dat};
\end{axis}
\end{tikzpicture}

(a) $t=0$
\end{minipage}
\begin{minipage}{0.32\textwidth}
\centering
\begin{tikzpicture}[remember picture]
\begin{axis}[compat = newest,xlabel={$x$},ylabel={$u$},enlarge x limits=false]
\addplot[black,smooth] table[x=x,y=y] {perlalp16xy.dat};
\addplot[mark = *,only marks,black] coordinates
{(63.03,0.01)};
\addplot[black,only marks,mark = |] table[x=x,y=c] {perlalp16xyd.dat};
\end{axis}
\end{tikzpicture}

(b) $t=16$
\end{minipage}
\begin{minipage}{0.32\textwidth}
\centering
\begin{tikzpicture}[remember picture]
\begin{axis}[compat = newest,xlabel={$x$},ylabel={$u$},enlarge x limits=false]
\addplot[black,smooth] table[x=x,y=y] {perlalp32xy.dat};
\addplot[mark = *,only marks,black] coordinates
{(63.78,-0.16)};
\addplot[black,only marks,mark = |] table[x=x,y=c] {perlalp32xyd.dat};
\end{axis}
\end{tikzpicture}

(c) $t=32$
\end{minipage}
\caption{ 
The loop and anti-loop soliton solution: solid lines correspond to the numerical solution 
obtained from numerical solution shown in \cref{fig_per_lal_sge} via discrete hodograph transformation~\cref{eq_dhodom} ($ \xi = 1.2, $ $ \Delta t = 0.1$, $ S = 80 $, and $ K = 257 $). 
Bars are for every four sample points. The symbol `$\bullet$' denotes the base point $ \left( \xd{m}{0} ,\ud{m}{0} \right) $.}
\label{fig_per_lal_spe}
\end{minipage}

\begin{minipage}{\textwidth}
\pgfplotsset{width=4cm}
\centering
\begin{minipage}{0.32\textwidth}
\centering
\begin{tikzpicture}[remember picture]
\begin{axis}[compat = newest,xlabel={$s$},ylabel={$\theta$},enlarge x limits=false]
\addplot[black,smooth] table {perlalp0arg.dat};
\end{axis}
\end{tikzpicture}

(a) $t=0$
\end{minipage}
\begin{minipage}{0.32\textwidth}
\centering
\begin{tikzpicture}[remember picture]
\begin{axis}[compat = newest,xlabel={$s$},ylabel={$\theta$},enlarge x limits=false]
\addplot[black,smooth] table {perlalp16arg.dat};
\end{axis}
\end{tikzpicture}

(b) $t=16$
\end{minipage}
\begin{minipage}{0.32\textwidth}
\centering
\begin{tikzpicture}[remember picture]
\begin{axis}[compat = newest,xlabel={$s$},ylabel={$\theta$},enlarge x limits=false]
\addplot[black,smooth] table {perlalp32arg.dat};
\end{axis}
\end{tikzpicture}

(c) $t=32$
\end{minipage}
\caption{ 
The loop and anti-loop soliton solution: solid lines correspond to the numerical solutions of the standard DVDM~\cref{eq_sg_dvdm_cd} for the SG equation  ($ \xi = 1.2, $ $ \Delta t = 0.1$, $ S = 80 $, and $ K = 257 $). }
\label{fig_per_lal_sge}
\end{minipage}
\end{figure}

\begin{figure}[htbp]
\begin{minipage}{\textwidth}
\pgfplotsset{width=4cm}
\centering
\begin{minipage}{0.32\textwidth}
\centering
\begin{tikzpicture}[remember picture]
\begin{axis}[compat = newest,xlabel={$x$},ylabel={$u$},enlarge x limits=false]
\addplot[black,smooth] table[x=x,y=y] {perlalpAD0xy.dat};
 \addplot[mark = *,only marks,black] coordinates
{(70.31,-0.00)};
\addplot[black,only marks,mark = |] table[x=x,y=c] {perlalpAD0xyd.dat};
\end{axis}
\end{tikzpicture}

(a) $t=0$
\end{minipage}
\begin{minipage}{0.32\textwidth}
\centering
\begin{tikzpicture}[remember picture]
\begin{axis}[compat = newest,xlabel={$x$},ylabel={$u$},enlarge x limits=false]
\addplot[black,smooth] table[x=x,y=y] {perlalpAD4xy.dat};
\addplot[mark = *,only marks,black] coordinates
{(3.91,-0.00)};
\addplot[black,only marks,mark = |] table[x=x,y=c] {perlalpAD4xyd.dat};
\end{axis}
\end{tikzpicture}

(b) $t=400$
\end{minipage}
\begin{minipage}{0.32\textwidth}
\centering
\begin{tikzpicture}[remember picture]
\begin{axis}[compat = newest,xlabel={$x$},ylabel={$u$},enlarge x limits=false]
\addplot[black,smooth] table[x=x,y=y] {perlalpAD8xy.dat};
\addplot[mark = *,only marks,black] coordinates
{(11.83,2.70)};
\addplot[black,only marks,mark = |] table[x=x,y=c] {perlalpAD8xyd.dat};
\end{axis}
\end{tikzpicture}

(c) $t=800$
\end{minipage}
\caption{ 
The loop and anti-loop soliton solution: solid lines correspond to the numerical solution 
obtained from numerical solution shown in \cref{fig_per_lal_ADsge} via the discrete hodograph transformation~\cref{eq_dhodom} ($ \xi = 1.2, $ $ \Delta t = 0.1$, $ S = 80 $, and $ K = 129 $). 
Bars are for every four sample points. The symbol `$\bullet$' denotes the base point $ \left( \xd{m}{0} ,\ud{m}{0} \right) $.}
\label{fig_per_lal_ADspe}
\end{minipage}

\begin{minipage}{\textwidth}
\pgfplotsset{width=4cm}
\centering
\begin{minipage}{0.32\textwidth}
\centering
\begin{tikzpicture}[remember picture]
\begin{axis}[compat = newest,xlabel={$s$},ylabel={$\theta$},enlarge x limits=false]
\addplot[black,smooth] table {perlalpAD0arg.dat};
\end{axis}
\end{tikzpicture}

(a) $t=0$
\end{minipage}
\begin{minipage}{0.32\textwidth}
\centering
\begin{tikzpicture}
\begin{axis}[compat = newest,xlabel={$s$},ylabel={$\theta$},enlarge x limits=false]
\addplot[black,smooth] table {perlalpAD4arg.dat};
\end{axis}
\end{tikzpicture}

(b) $t=400$
\end{minipage}
\begin{minipage}{0.32\textwidth}
\centering
\begin{tikzpicture}
\begin{axis}[compat = newest,xlabel={$s$},ylabel={$\theta$},enlarge x limits=false]
\addplot[black,smooth] table {perlalpAD8arg.dat};
\end{axis}
\end{tikzpicture}

(c) $t=800$
\end{minipage}
\caption{ 
The loop and anti-loop soliton solution: solid lines correspond to the numerical solutions of the average-difference method~\cref{eq_sg_dvdm} for the SG equation ($ \xi = 1.2, $ $ \Delta t = 0.1$, $ S = 80 $, and $ K = 129 $). }
\label{fig_per_lal_ADsge}
\end{minipage}
\end{figure}

\subsection{Numerical experiments for exact periodic solutions}
\label{subsec_periodic}

Finally, we confirm the validity of the proposed method using 
the exact periodic traveling wave solutions. 
Below we review several exact solutions under the periodic boundary condition devised by Parkes~\cite{P2008}
(see also Matsuno~\cite{M2008} for similar periodic solutions). 
All numerical experiments are done under the condition $ \Delta t = 0.1 $ and $ K = 65$. 

\subsubsection{Periodic hump solution}

One of the periodic traveling-wave solutions can be written as follows: 
\begin{equation}
\begin{cases}
{\displaystyle x(\tau,s) = v \tau + x_0 - s - \frac{1}{\alpha} f(\tau) + \frac{2}{\alpha} E (\alpha s + f(\tau) \mid \xi ),}\\[5pt]
{\displaystyle u(\tau,s) = \pm \frac{2 \sqrt{\xi} }{ \alpha } \mathrm{cn} ( \alpha s + f(\tau) \mid \xi ),}
\end{cases}
\label{solution_periodic_hump}
\end{equation}
where $ v > 0 $ represents the velocity of the traveling-wave, 
$ \xi \in (0, 1/2) $ denotes the parameter related to the shape of the wave, 
and $ \alpha := \sqrt{(1  - 2 \xi )/v} $. 
Here, $ \mathrm{cn} ( w | \xi ) $ is a Jacobian elliptic function, and $ E ( w | \xi ) $ is the elliptic integral of the second kind 
(with the notation as in~\cite{Handbook}). 
We also use the Jacobian elliptic functions $ \mathrm{sn} (w | \xi ) $ and $ \mathrm{dn} (w| \xi ) $, 
and $ K ( w | m ) $ denotes the elliptic integral of the first kind. 

The traveling-wave solution~\cref{solution_periodic_hump} satisfies the moving grid form~\cref{eq_spemg} of the SP equation 
for any smooth function $ f : \R_+ \to \R$. 
However, in order to transform it
to the solution of the SG equation via the hodograph transformation, 
the solution~\cref{solution_periodic_hump} must satisfy the condition~\cref{abp}. 
The condition holds if and only if the function $ f $ satisfies the differential equation 
\begin{equation*}
\frac{ \rd f }{ \rd \tau } = - \frac{1}{\alpha}. 
\end{equation*}
Hence, we define $ f ( \tau ) = - \tau /\alpha $. 
Thus, we obtain the concrete form of the traveling wave solution 
\begin{equation}
\begin{cases}
{\displaystyle x(\tau,s) = v \tau + x_0 - s +\frac{1}{\alpha^2} \tau + \frac{2}{\alpha} E \left(\alpha s- \frac{\tau}{\alpha}  \relmiddle{|} \xi \right),} \\[5pt]
{\displaystyle u(\tau,s) = \pm \frac{2 \sqrt{ \xi } }{ \alpha } \mathrm{cn} \left(\alpha s- \frac{\tau}{\alpha}  \relmiddle{|} \xi \right). }
\end{cases}
\label{eq_spe_perhump}
\end{equation}
The corresponding solution of the SG equation is given by 
\begin{equation}
\theta(\tau,s) = \arctan \frac{ \mp 2 \sqrt{ \xi } \, \mathrm{sn} \left(\alpha s- \frac{\tau}{\alpha}  \relmiddle{|} \xi \right) \mathrm{dn} \left(\alpha s- \frac{\tau}{\alpha}  \relmiddle{|} \xi \right)  }{-1+2 \left( \mathrm{dn} \left(\alpha s- \frac{\tau}{\alpha}  \relmiddle{|} \xi \right) \right)^2 }.
\label{eq_sge_perhump}
\end{equation}
We computed $\dthe{0}{}$ from this expression
to prepare accurate initial data to capture the exact solution.

In \cref{fig_per_ph_spe}, 
we show the numerical solution corresponding to the periodic hump solution~\cref{eq_spe_perhump} obtained by the proposed method.
The numerical solution of the SG equation
corresponding to~\cref{eq_sge_perhump} is shown in \cref{fig_per_ph_sge}. 
We see the numerical solution beautifully keeps the original shape
after time integration.

\begin{figure}
\begin{minipage}{\textwidth}
\pgfplotsset{width=4cm}
\centering
\begin{minipage}{0.3\textwidth}
\centering
\begin{tikzpicture}[remember picture]
\begin{axis}[compat = newest,xlabel={$x$},ylabel={$u$},enlarge x limits=false]
\addplot[black,smooth] table[x=x,y=y] {perhump0xyd.dat};
\addplot[mark = *,only marks,black] coordinates
{(14.0329,1.4132)};
\addplot[black,only marks,mark = |] table[x=x,y=c] {perhump0xyd.dat};
\end{axis}
\end{tikzpicture}
(a) $t=0$
\end{minipage}
\begin{minipage}{0.3\textwidth}
\centering
\begin{tikzpicture}[remember picture]
\begin{axis}[compat = newest,xlabel={$x$},ylabel={$u$},enlarge x limits=false]
\addplot[black,smooth] table[x=x,y=y] {perhump5xyd.dat};
\addplot[mark = *,only marks,black] coordinates
{(11.5266,1.3373)};
\addplot[black,only marks,mark = |] table[x=x,y=c] {perhump5xyd.dat};
\end{axis}
\end{tikzpicture}
(b) $t=5$
\end{minipage}
\begin{minipage}{0.3\textwidth}
\centering
\begin{tikzpicture}[remember picture]
\begin{axis}[compat = newest,xlabel={$x$},ylabel={$u$},enlarge x limits=false]
\addplot[black,smooth] table[x=x,y=y] {perhump10xyd.dat};
\addplot[mark = *,only marks,black] coordinates
{(9.0842,1.0951)};
\addplot[black,only marks,mark = |] table[x=x,y=c] {perhump10xyd.dat};
\end{axis}
\end{tikzpicture}
(c) $t=10$
\end{minipage}
\caption{ 
The periodic hump solution: solid lines correspond to the numerical solution of~\cref{eq_spe_perhump} with $ v = 1, x_0 = 0, \xi = 0.25 $. 
Bars are for sample points. The symbol `$\bullet$' denotes the base point $ \left( \xd{m}{0} ,\ud{m}{0} \right) $.}
\label{fig_per_ph_spe}
\end{minipage}

\begin{minipage}{\textwidth}
\pgfplotsset{width=4cm}
\centering
\begin{minipage}{0.3\textwidth}
\centering
\begin{tikzpicture}[remember picture]
\begin{axis}[compat = newest,xlabel={$s$},ylabel={$\theta$},enlarge x limits=false]
\addplot[black,smooth] table {perhump0arg.dat};
\end{axis}
\end{tikzpicture}
(a) $t=0$
\end{minipage}
\begin{minipage}{0.3\textwidth}
\centering
\begin{tikzpicture}[remember picture]
\begin{axis}[compat = newest,xlabel={$s$},ylabel={$\theta$},enlarge x limits=false]
\addplot[black,smooth] table {perhump5arg.dat};
\end{axis}
\end{tikzpicture}
(b) $ t = 5$
\end{minipage}
\begin{minipage}{0.3\textwidth}
\centering
\begin{tikzpicture}[remember picture]
\begin{axis}[compat = newest,xlabel={$s$},ylabel={$\theta$},enlarge x limits=false]
\addplot[black,smooth] table {perhump10arg.dat};
\end{axis}
\end{tikzpicture}
(c) $ t = 10$
\end{minipage}
\caption{ 
The periodic hump solution: solid lines correspond to the numerical solutions of~\cref{eq_sge_perhump} with $ v = 1, x_0 = 0, \xi = 0.25 $.}
\label{fig_per_ph_sge}
\end{minipage}
\end{figure}

\subsubsection{Periodic upright loop solution}

\begin{figure}[htbp]
\begin{minipage}{\textwidth}
\pgfplotsset{width=4cm}
\centering
\begin{minipage}{0.3\textwidth}
\centering
\begin{tikzpicture}[remember picture]
\begin{axis}[compat = newest,xlabel={$x$},ylabel={$u$},enlarge x limits=false]
\addplot[black,smooth] table[x=x,y=y] {perloop0xyd5.dat};
\addplot[black,smooth] table[x=x,y=y] {perloop0xyd1.dat};
\addplot[black,smooth] table[x=x,y=y] {perloop0xyd2.dat};
\addplot[black,smooth] table[x=x,y=y] {perloop0xyd3.dat};
\addplot[black,smooth] table[x=x,y=y] {perloop0xyd4.dat};
\addplot[mark = *,only marks,black] coordinates
{( 1.79,1.82)};
\addplot[black,only marks,mark = |] table[x=x,y=c] {perloop0xyd.dat};
\end{axis}
\end{tikzpicture}
(a) $t=0$
\end{minipage}
\begin{minipage}{0.3\textwidth}
\centering
\begin{tikzpicture}[remember picture]
\begin{axis}[compat = newest,xlabel={$x$},ylabel={$u$},enlarge x limits=false]
\addplot[black,smooth] table[x=x,y=y] {perloop5xyd1.dat};
\addplot[black,smooth] table[x=x,y=y] {perloop5xyd2.dat};
\addplot[black,smooth] table[x=x,y=y] {perloop5xyd3.dat};
\addplot[mark = *,only marks,black] coordinates
{(0.78,1.80)};
\addplot[black,only marks,mark = |] table[x=x,y=c] {perloop5xyd.dat};
\end{axis}
\end{tikzpicture}
(b) $t=5$
\end{minipage}
\begin{minipage}{0.3\textwidth}
\centering
\begin{tikzpicture}[remember picture]
\begin{axis}[compat = newest,xlabel={$x$},ylabel={$u$},enlarge x limits=false]
\addplot[black,smooth] table[x=x,y=y] {perloop10xyd1.dat};
\addplot[black,smooth] table[x=x,y=y] {perloop10xyd2.dat};
\addplot[black,smooth] table[x=x,y=y] {perloop10xyd3.dat};
\addplot[mark = *,only marks,black] coordinates
{(1.76,1.69)};
\addplot[black,only marks,mark = |] table[x=x,y=c] {perloop10xyd.dat};
\end{axis}
\end{tikzpicture}
(c) $t=10$
\end{minipage}
\caption{ 
The periodic loop solution: solid lines correspond to the numerical solution of~\cref{eq_spe_perloop} with $ v = 1, x_0 = 0, \xi = 0.75 $. 
Bars are for sample points. The symbol `$\bullet$' denotes the base point $ \left( \xd{m}{0} ,\ud{m}{0} \right) $.}
\label{fig_per_pl_spe}
\end{minipage}

\begin{minipage}{\textwidth}
\pgfplotsset{width=4cm}
\centering
\begin{minipage}{0.3\textwidth}
\centering
\begin{tikzpicture}[remember picture]
\begin{axis}[compat = newest,xlabel={$s$},ylabel={$\theta$},enlarge x limits=false]
\addplot[black,smooth] table {perloop0arg.dat};
\end{axis}
\end{tikzpicture}
(a) $t=0$
\end{minipage}
\begin{minipage}{0.3\textwidth}
\centering
\begin{tikzpicture}[remember picture]
\begin{axis}[compat = newest,xlabel={$s$},ylabel={$\theta$},enlarge x limits=false]
\addplot[black,smooth] table {perloop5arg.dat};
\end{axis}
\end{tikzpicture}
(b) $ t = 5$
\end{minipage}
\begin{minipage}{0.3\textwidth}
\centering
\begin{tikzpicture}[remember picture]
\begin{axis}[compat = newest,xlabel={$s$},ylabel={$\theta$},enlarge x limits=false]
\addplot[black,smooth] table {perloop10arg.dat};
\end{axis}
\end{tikzpicture}
(c) $ t = 10$
\end{minipage}
\caption{ 
The periodic loop solution: solid lines correspond to the numerical solutions of the SG equation corresponding to the solution~\cref{eq_spe_perloop} of the SP equation with $ v = 1, x_0 = 0, \xi = 0.75 $.}
\label{fig_per_pl_sge}
\end{minipage}
\end{figure}

Similarly to the above,
the periodic upright loop solutions can be written as
\begin{equation}
\begin{cases}
{\displaystyle x(\tau,s) = -v \tau + x_0 - \xi \alpha \left( \alpha s + f(\tau)  \right) + \frac{2}{\xi \alpha} E (\alpha s + f(\tau) \mid \xi ),}\\[5pt]
{\displaystyle u(\tau,s) = \pm \frac{2  }{ \xi \alpha } \mathrm{dn} ( \alpha s + f(\tau) \mid \xi ),}
\end{cases}
\label{solution_periodic_loop}
\end{equation}
where in this case $ \xi \in (0, 1) $
and $ \alpha := \sqrt{ (2- \xi )/ ( \xi^2 v)} $. 

Again the solution~\cref{solution_periodic_loop} must satisfy the condition~\cref{abp}, which implies
\begin{equation*}
\frac{ \rd f }{ \rd \tau } = - \frac{1}{\xi \alpha^2}. 
\end{equation*}
Hence, we define $ f (t) = - \tau / ( \xi \alpha^2 )$
to find
\begin{equation}
\begin{cases}
{\displaystyle x(\tau,s) = -v \tau + x_0 - \xi \alpha^2  s + \frac{1}{\alpha} \tau  + \frac{2}{ \xi \alpha} E \left(\alpha s - \frac{ \tau }{ \xi \alpha^2} \relmiddle{|} \xi \right),}\\[5pt]
{\displaystyle u(\tau,s) = \pm \frac{2  }{ \xi \alpha } \mathrm{dn} \left( \alpha s - \frac{ \tau }{ \xi \alpha^2} \relmiddle{|} \xi  \right).}\end{cases}
\label{eq_spe_perloop}
\end{equation}
We prepared the initial data from this expression.
The results well captures the exact solution (\cref{fig_per_pl_spe} and \cref{fig_per_pl_sge}).

\subsubsection{Periodic solution comprising alternating upright and inverted bells}

\begin{figure}
\begin{minipage}{\textwidth}
\pgfplotsset{width=4cm}
\centering
\begin{minipage}{0.3\textwidth}
\centering
\begin{tikzpicture}[remember picture]
\begin{axis}[compat = newest,xlabel={$x$},ylabel={$u$},enlarge x limits=false]
\addplot[black,smooth] table[x=x,y=y] {peralt0xyd1.dat};
\addplot[black,smooth] table[x=x,y=y] {peralt0xyd2.dat};
\addplot[black,smooth] table[x=x,y=y] {peralt0xyd3.dat};
\addplot[mark = *,only marks,black] coordinates
{( 2.70,2.45)};
\addplot[black,only marks,mark = |] table[x=x,y=c] {peralt0xyd.dat};
\end{axis}
\end{tikzpicture}
(a) $t=0$
\end{minipage}
\begin{minipage}{0.3\textwidth}
\centering
\begin{tikzpicture}[remember picture]
\begin{axis}[compat = newest,xlabel={$x$},ylabel={$u$},enlarge x limits=false]
\addplot[black,smooth] table[x=x,y=y] {peralt5xyd1.dat};
\addplot[black,smooth] table[x=x,y=y] {peralt5xyd2.dat};
\addplot[black,smooth] table[x=x,y=y] {peralt5xyd3.dat};
\addplot[black,smooth] table[x=x,y=y] {peralt5xyd4.dat};
\addplot[black,smooth] table[x=x,y=y] {peralt5xyd5.dat};
\addplot[mark = *,only marks,black] coordinates
{(0.05,0.75)};
\addplot[black,only marks,mark = |] table[x=x,y=c] {peralt5xyd.dat};
\end{axis}
\end{tikzpicture}
(b) $t=5$
\end{minipage}
\begin{minipage}{0.3\textwidth}
\centering
\begin{tikzpicture}[remember picture]
\begin{axis}[compat = newest,xlabel={$x$},ylabel={$u$},enlarge x limits=false]
\addplot[black,smooth] table[x=x,y=y] {peralt10xyd1.dat};
\addplot[black,smooth] table[x=x,y=y] {peralt10xyd2.dat};
\addplot[black,smooth] table[x=x,y=y] {peralt10xyd3.dat};
\addplot[mark = *,only marks,black] coordinates
{(1.71,-1.17)};
\addplot[black,only marks,mark = |] table[x=x,y=c] {peralt10xyd.dat};
\end{axis}
\end{tikzpicture}
(c) $t=10$
\end{minipage}
\caption{ 
The periodic solution comprising alternating upright and inverted bells: solid lines correspond to the numerical solution of~\cref{eq_spe_peralt} with $ v = 1, x_0 = 0, \xi = 0.75 $. 
Bars are for sample points. The symbol `$\bullet$' denotes the base point $ \left( \xd{m}{0} ,\ud{m}{0} \right) $.}
\label{fig_per_pa_spe}
\end{minipage}

\begin{minipage}{\textwidth}
\pgfplotsset{width=4cm}
\centering
\begin{minipage}{0.3\textwidth}
\centering
\begin{tikzpicture}[remember picture]
\begin{axis}[compat = newest,xlabel={$s$},ylabel={$\theta$},enlarge x limits=false]
\addplot[black,smooth] table {peralt0arg.dat};
\end{axis}
\end{tikzpicture}
(a) $t=0$
\end{minipage}
\begin{minipage}{0.3\textwidth}
\centering
\begin{tikzpicture}[remember picture]
\begin{axis}[compat = newest,xlabel={$s$},ylabel={$\theta$},enlarge x limits=false]
\addplot[black,smooth] table {peralt5arg.dat};
\end{axis}
\end{tikzpicture}
(b) $ t = 5$
\end{minipage}
\begin{minipage}{0.3\textwidth}
\centering
\begin{tikzpicture}[remember picture]
\begin{axis}[compat = newest,xlabel={$s$},ylabel={$\theta$},enlarge x limits=false]
\addplot[black,smooth] table {peralt10arg.dat};
\end{axis}
\end{tikzpicture}
(c) $ t = 10$
\end{minipage}
\caption{ 
The periodic solution comprising alternating upright and inverted bells: solid lines correspond to the numerical solutions of the SG equation corresponding to the solution~\cref{eq_spe_peralt} of the SP equation with $ v = 1, x_0 = 0, \xi = 0.75 $.}
\label{fig_per_pa_sge}
\end{minipage}
\end{figure}

The periodic solutions comprising alternating upright and inverted bells can be written as
\begin{equation}
\begin{cases}
{\displaystyle x(\tau,s) = -v \tau + x_0 - s - \frac{1}{\alpha} f(\tau) + \frac{2}{\alpha} E \left( \alpha s + f(\tau) \relmiddle{|} \xi \right), } \\[5pt]
{\displaystyle u(\tau,s) = \pm \frac{2 \sqrt{ \xi } }{ \alpha } \mathrm{cn} \left( \alpha s + f(\tau) \relmiddle{|} \xi \right),}
\end{cases}
\label{solution_periodic_alt}
\end{equation}
where $ \xi \in (1/2, 1) $ 
and $ \alpha := \sqrt{ (2 \xi -1)/v} $. 

Similarly to the above, from the constraint
\begin{equation*}
\frac{ \rd f }{ \rd \tau } = - \frac{1}{\alpha},
\end{equation*}
we define $ f (t) = - \tau / \alpha $, and  obtain
\begin{equation}
\begin{cases}
{\displaystyle x(\tau,s) = v \tau + x_0 - s +\frac{1}{\alpha^2} \tau + \frac{2}{\alpha} E \left(\alpha s- \frac{\tau}{\alpha}  \relmiddle{|} \xi \right),} \\[5pt]
{\displaystyle u(\tau,s) = \pm \frac{2 \sqrt{ \xi } }{ \alpha } \mathrm{cn} \left(\alpha s- \frac{\tau}{\alpha}  \relmiddle{|} \xi \right). }
\end{cases}
\label{eq_spe_peralt}
\end{equation}
The initial data was prepared from this expression.
The results are shown in \cref{fig_per_pa_spe} and \cref{fig_per_pa_sge},
which are again quite well.

\section{Concluding Remarks}
\label{sec_conclusion}

In this paper, 
we established a new self-adaptive moving mesh method
for the short pulse equation, employing
several ideas from integrable systems,
moving mesh methods, and structure-preserving methods;
none of which is truly dispensable.
We hope that, in the reverse direction, this work
will invoke new discussions in each research field.

Some possible future works are in order.
Feng--Maruno--Ohta~\cite{FMO2014} showed that 
the short pulse equation is also associated with the coupled dispersionless (CD) equation~\cite{KO1994_1,HT1994}, 
and this relation can be used for the construction of another self-adaptive moving mesh method. 
The present authors have already confirmed that 
a similar self-adaptive numerical scheme can be constructed based on this formulation~\cite{SMF2016}. 
Furthermore, we can extend the approach to wide range of similar PDEs
including the complex SP equation, whose exact solutions were recently found by Ling--Feng--Zhu~\cite{LFZ2016}. 
Another interesting topic is, as stated in Introduction,
the challenge of extending standard moving mesh methods
to the PDEs including the term $ u_{tx} $.





\end{document}